\documentclass[11pt]{article}
\usepackage{amsthm}
\usepackage{amssymb}
\usepackage{amsmath,enumerate}
\usepackage{comment}
\usepackage{thm-restate}
\usepackage{url} 
\usepackage{hyperref}
\usepackage{graphicx}
\usepackage{subfigure}
\usepackage[noabbrev,capitalise]{cleveref}
\usepackage[affil-it]{authblk}
\usepackage{color}
\usepackage[normalem]{ulem}

\usepackage[margin=1in]{geometry}

\renewcommand{\v}{\textup{\textsf{v}}}

\theoremstyle{plain}
\newtheorem{thm}{Theorem}[section]
\newtheorem{lem}[thm]{Lemma}

\newtheorem{cor}[thm]{Corollary}

\newtheorem{rem}[thm]{Remark}
\newtheorem{conj}[thm]{Conjecture}

{\noindent \emph{Proof.} {}{#1}{}}{\hfill
	$\Diamond$\vspace{1em}}

\theoremstyle{plain} 
\newcommand{\thistheoremname}{}
\newtheorem{genericthm}[section]{\thistheoremname}

\theoremstyle{definition}

\def\less{\setminus}

\newcounter{counter}

\def\ksix{\mathcal{K}_6^{-4}}
\def\kst{\mathcal{K}_7^{-3}}
\def\ket{\mathcal{K}_8^{-3}}
\def\kef{\mathcal{K}_8^{-4}}
\def\se{\succcurlyeq}

\def\dfn#1{{\sl #1}}

\title{Every graph with no $\kef$ minor is $7$-colorable}
\author{Michael Lafferty\thanks{Department  of Mathematics, University of Central Florida, Orlando, FL 32816, USA. Supported   by NSF grant DMS-2153945. Email: michaelmlafferty@Knights.ucf.edu.}\hskip 1cm Zi-Xia Song\thanks{Department  of Mathematics, University of Central Florida, Orlando, FL 32816, USA. Supported by  NSF grant    DMS-2153945. Email: Zixia.Song@ucf.edu.}}
 
\date{August 22, 2022}
\begin{document}
\maketitle
\begin{center}
	\emph{Dedicated to the memory of Robin Thomas on his 60th birthday}
\end{center}

\begin{abstract}
 
 Hadwiger's Conjecture from 1943 states that every graph with no $K_{t}$ minor is $(t-1)$-colorable;   it remains wide open for all $t\ge 7$.  For positive integers $t$ and $s$, let $\mathcal{K}_t^{-s}$ denote the family of graphs obtained from the complete graph $K_t$ by removing $s$ edges. We say that a graph $G$ has no  $\mathcal{K}_t^{-s}$ minor  if it has no $H$ minor for every $H\in  \mathcal{K}_t^{-s}$. Jakobsen in 1971 proved that every   graph  with no   $\mathcal{K}_7^{-2}$   minor is $6$-colorable.   In this paper we consider the next step and  prove that every graph with no $\kef$ minor  is $7$-colorable.  Our result implies that $H$-Hadwiger's Conjecture,   suggested by Paul Seymour in 2017,  is true for every graph $H$ on eight vertices such that  the complement of $H$ has maximum degree at least four,   a perfect matching, a triangle  and a cycle of length   four.  Our proof utilizes  an extremal function for $\kef$ minors obtained in this paper,   generalized Kempe chains of contraction-critical graphs    by Rolek and the second author, and 
 the method for finding     $K_7$ minors  from three different   $K_5$ subgraphs by Kawarabayashi and Toft;  this  method  was first   developed  by Robertson, Seymour and Thomas in 1993 to prove Hadwiger's Conjecture for $t=6$. 
\end{abstract}

\baselineskip 16pt

\section{Introduction}

All graphs in this paper are finite and simple.   For a graph $G$ we use $|G|$, $e(G)$, $\delta (G)$, $\Delta(G)$, $\alpha(G)$, $\chi(G)$ to denote the number
of vertices, number of edges,   minimum degree, maximum degree,  independence number, and chromatic number  of $G$, respectively.  The \dfn{complement} of $G$ is denoted by $\overline{G}$.  For any positive integer k, we write  $[k]$ for the set $\{1, \ldots, k\}$. A graph  $H$ is a \dfn{minor} of a graph $G$ if  $H$ can be
 obtained from a subgraph of $G$ by contracting edges.  We write $G\se H$ if 
$H$ is a minor of $G$.
In those circumstances we also say that  $G$ has an  \dfn{$H$ minor}.  For positive integers  $t, s$, we use $\mathcal{K}_t^{-s}$ to denote  the family of graphs obtained from the complete graph $K_t$ by deleting $s$ edges. When $s\le 2$, we use $K_t^-$ to denote the unique graph  in $\mathcal{K}_t^{-1}$; and $K_t^=$ and  $K_t^<$ to denote the graphs obtained from $K_t$ by deleting two independent edges and two adjacent edges, respectively. Note that $\mathcal{K}_t^{-2}=\{K_t^=, K_t^<\}$. 
We say that a graph $G$ has \dfn{no  $\mathcal{K}_t^{-s}$ minor}  if it has no $H$ minor for every $H\in  \mathcal{K}_t^{-s}$; and $G$ has a  $\mathcal{K}_t^{-s}$ minor, otherwise. We write $G\se \mathcal{K}_t^{-s}$ if 
$G$ has a  $\mathcal{K}_t^{-s}$ minor.\medskip

Our work is motivated by   Hadwiger's Conjecture~\cite{Had43}, which is perhaps the most famous conjecture in graph theory.

\begin{conj}[Hadwiger's Conjecture~\cite{Had43}]\label{HC} Every graph with no $K_t$ minor is $(t-1)$-colorable. 
\end{conj}

\cref{HC} is  trivially true for $t\le3$, and reasonably easy for $t=4$, as shown independently by Hadwiger~\cite{Had43} and Dirac~\cite{Dirac52}. However, for $t\ge5$, Hadwiger's Conjecture implies the Four Color Theorem~\cite{AH77,AHK77}.   Wagner~\cite{Wagner37} proved that the case $t=5$ of Hadwiger's Conjecture is, in fact, equivalent to the Four Color Theorem, and the same was shown for $t=6$ by Robertson, Seymour and  Thomas~\cite{RST}. Despite receiving considerable attention over the years, Hadwiger's Conjecture remains wide open for all $t\ge 7$,  and is    considered among the most important problems in graph theory and has motivated numerous developments in graph coloring and graph minor theory.  Proving that graphs with no $K_7$ minor are $6$-colorable is thus the first case of Hadwiger's  Conjecture that is still open.
It  is not even known yet whether  every graph with no $K_7$ minor is $7$-colorable; Rolek, Thomas and the second author~\cite{RST22} proved that every $8$-contraction-critical graph with no $K_7$ minor has at most one vertex of degree eight. Jakobsen~\cite{Jakobsen71b}  in 1971 proved that every graph with no $K_7$   minor is $9$-colorable. 
  Kawarabayashi and Toft~\cite{KT05} in 2005 proved that every graph with no $K_7$ or $K_{4,\, 4}$ minor is $6$-colorable.   Recently, Albar and Gon\c calves~\cite{AG18}  proved that 
every graph with no $K_7$ minor is $8$-colorable, and every graph with no $K_8$ minor is $10$-colorable. Their proofs are computer-assisted;      Rolek and the second author~\cite{RolekSong17a} then gave  computer-free proofs of  their results, and further  proved that  every graph with no $K_9$ minor is $12$-colorable, and every graph with no $K_t$ minor is $(2t-6)$-colorable for all $t\ge10$ if  Conjecture~5.1 in \cite{RolekSong17a} holds. \medskip

Until very recently  the best known upper bound on the chromatic number of graphs with no $K_t$ minor  is $O(t (\log t)^{1/2})$,  obtained independently by Kostochka~\cite{Kostochka82,Kostochka84} and Thomason~\cite{Thomason84}, while  Norin, Postle and the second  author~\cite{NPS20} improved    the frightening $(\log t)^{1/2}$ term to $(\log t)^{1/4}$. The current record 
 is $O(t\log \log t)$ due to
Delcourt and Postle~\cite{DelcourtPostle}.   K\"{u}hn  and Osthus~\cite{KuhOst03c} proved that Hadwiger's Conjecture is true for $C_4$-free graphs of sufficiently large chromatic number,  and for all graphs of girth at least $19$.  We refer the reader to   recent surveys~\cite{CV2020, K2015,Seymoursurvey} for further  background on Hadwiger's Conjecture. \medskip

Given the difficulty of Hadwiger's Conjecture, Paul Seymour in 2017 suggested the study of the following $H$-Hadwiger's Conjecture. 

\begin{conj}[$H$-Hadwiger's Conjecture]\label{HHC} For every graph $H$ on $t$ vertices, every graph with no $H$ minor is $(t-1)$-colorable. 
\end{conj}

It is worth noting that Wagner~\cite{Wagner60}  in 1960 initiated  the study of this type of problem   and proved that every graph with no $K_5^-$ minor is $4$-colorable;   Dirac~\cite{Dirac64a}  in 1964    proved that every graph with no $K_6^-$ minor is $5$-colorable;  and Jakobsen~\cite{Jakobsen71b} in 1971 proved that every graph with no $K_7^{-}$ minor is  $7$-colorable. Thus proving graphs with no $K_7^-$ minor is  $6$-colorable remains open.  Rolek and the second author~\cite{RolekSong17a} in 2017 proved that every graph with no $K_8^{-}$ minor is  $9$-colorable.  Very recently, Norin and Seymour~\cite{NorSey22} proved that every graph on $n$ vertices with independence number two has an $H$ minor, where $H$  is a graph with  $\lceil n/2\rceil$ vertices and at least $  0.98688\cdot {{|H|}\choose2}-o(n^2)$ edges. \medskip

Woodall~\cite{Woo01} studied a special case of $H$-Hadwiger's Conjecture in 2001 by excluding the complete bipartite graph $K_{s,t}$ minor and made the following  conjecture. 

 \begin{conj}[Woodall~\cite{Woo01}]\label{WC}  Every graph with no $K_{s,t}$ minor is $(s+t-1)$-colorable. 
\end{conj}
 
  In the same paper Woodall confirmed   Conjecture~\ref{WC}  for $s\in\{1,2\}$ without using the extremal function for $K_{s,t}$ minors.  The $s=2$ case also follows from the extremal function for $K_{2,t}$ minors by Myers~\cite{Myers03},  and Chudnovsky, Reed and Seymour~\cite{CRS11};   the $s=3$ case when $t\ge6300$  follows from the extremal function for $K_{3,t}$ minors by  Kostochka and Prince~\cite{KosPri10}. Asymptomatic bounds for the chromatic number of graphs with no   $K_{s,t}$ minor,  when $t$ is sufficiently larger than $s$, follow from the extremal functions for $K_{s,t}$ minors  by K\"uhn and Osthus~\cite{KuhOst05a},  and Kostochka and Prince~\cite{KosPri08}; in particular,  the extremal function by Kostochka and Prince  implies that every graph with no $K_{s,t}$ minor is $(3s+t-1)$-colorable when $t$ is sufficiently larger than $s$.    Years later Kostochka~\cite{Kos14}  proved that   Conjecture~\ref{WC} is true for $t>C(s\log s)^3$.\medskip

 Dirac  in 1964 began the study of a variation of $H$-Hadwiger's Conjecture in~\cite{Dirac64b} by  excluding  more than one forbidden minor simultaneously; he proved that every graph with no $\mathcal{K}_t^{-2}$ minor is  $(t-1)$-colorable for each $t\in\{5,6\}$. 
Jakobsen~\cite{Jakobsen71a} in 1971 proved that every graph with no $\mathcal{K}_7^{-2}$ minor is  $6$-colorable;   this implies that $H$-Hadwiger's Conjecture  is true  for all graphs $H$ on seven vertices such that $\Delta(\overline{H})\ge2$ and $\overline{H}$ has a matching of size two.   Recently, Rolek and the second author~\cite{RolekSong17a} proved that every graph with no $\mathcal{K}_8^{-2}$ minor is  $8$-colorable, and  Rolek~\cite{Rolek20} later proved that every graph with no $\mathcal{K}_9^{-2}$ minor is  $10$-colorable.  Proving that every graph with no $\mathcal{K}_8^{-2}$ minor is  $7$-colorable is still open.   We state the result of Jakobsen~\cite{Jakobsen71a} below.

\begin{thm}[Jakobsen~\cite{Jakobsen71a}]\label{t:K72} Every graph with no $\mathcal{K}_7^{-2}$ minor is  $6$-colorable.  In particular,   $H$-Hadwiger's Conjecture  is true for all graphs $H$ on seven vertices such that $\Delta(\overline{H})\ge2$ and $\overline{H}$ has a matching of size two.
\end{thm}

The purpose of this paper is to  consider the next step and prove the following main result.

\begin{restatable}{thm}{main}\label{t:main} Every graph with no $\kef$ minor is $7$-colorable.     \end{restatable}
 
There are 11 ways to delete four edges from $K_8$, including  four edges that form a perfect matching and four edges that are incident with  the same vertex. Let $H$ be a graph  on eight vertices such that $\Delta(\overline{H})\ge4$,  and $\overline{H}$ has a perfect matching, a triangle and a cycle of length four.     It is simple to check that  every graph with no   $H$ minor  has no $\kef$ minor.   Combining this with \cref{t:main} leads to the observation that $H$-Hadwiger's Conjecture   is true for all such graphs $H$. In particular, $H$-Hadwiger's Conjecture is true  for five graphs $H$ obtained from $K_8$ by deleting eight edges, where $\overline{H}$ is  given in Figure~\ref{f:SpecialH}.

\begin{figure}[htb]
\centering
\includegraphics[scale=0.3]{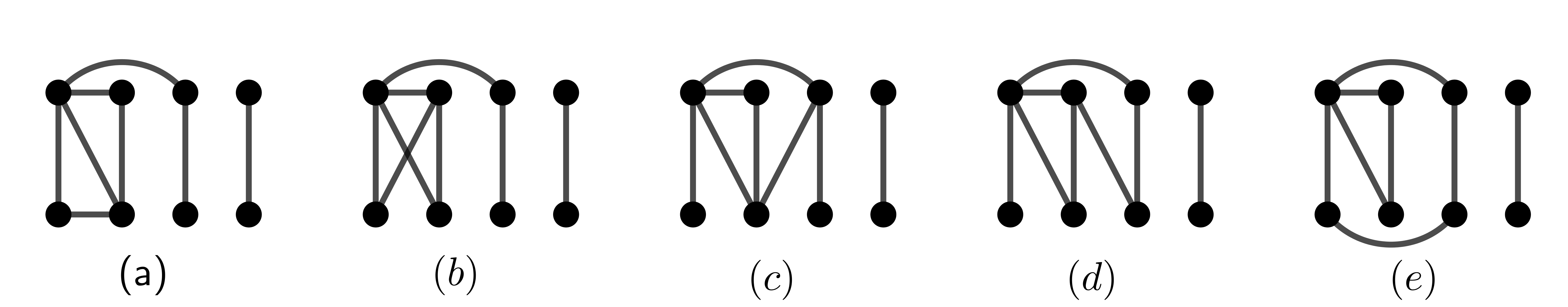}
\caption{Five non-isomorphic $\overline{H}$.}
\label{f:SpecialH}
\end{figure}

\begin{cor}\label{c:K84} $H$-Hadwiger's Conjecture  is true for all graphs $H$ on eight vertices such that $\Delta(\overline{H})\ge4$,  and $\overline{H}$ has a perfect matching,  a triangle and a cycle of length four. 
\end{cor}

 Our proof of \cref{t:main} utilizes  an extremal function for $\kef$ minors (see \cref{t:exfun}),  generalized Kempe chains of contraction-critical graphs obtained  by Rolek and the second author~\cite{RolekSong17a} (see \cref{l:wonderful}), and the method   for finding  $K_7$ minors  from three different   $K_5$ subgraphs by Kawarabayashi and Toft~\cite{KT05} (see \cref{t:goodpaths}).  \medskip
 

\begin{restatable}{thm}{exfun}\label{t:exfun} Every graph on $n\ge 8$ vertices with at least $4.5n-12$ edges has a $\kef$ minor. \end{restatable}

   \cref{t:exfun} is best possible in the sense that every $(K_{2,2,2,2}, 4)$-cockade (see the definition of an  $(H, k)$-cockade   in \cite{SongThomas2006}) on $n$ vertices has $4.5n-12$ edges but no $\ket$ minor. \medskip

 This paper is organized as follows. In the next section, we introduce the necessary definitions and collect several tools which we will need later on. In Section~\ref{s:coloring}, we prove \cref{t:main}. In Section~\ref{s:exfun}, we prove \cref{t:exfun} using more involved arguments.\medskip 
 
 \section{Notation and tools}

 Let $G$ be a graph.   If $x,y$ are adjacent
vertices of   $G$, then we denote by $G/xy$ the graph obtained from $G$
by contracting the edge $xy$ and deleting all resulting parallel
edges. We simply write $G/e$ if $e=xy$.  If $u,v$ are distinct nonadjacent vertices of   $G$, then by
$G+uv$ we denote the graph obtained from $G$ by adding an edge
with ends $u$ and $v$.  If $u,v$ are adjacent or equal, then we define
$G+uv$ to be $G$.  Similarly, if  $M\subseteq E(G)\cup E(\overline{G})$, then   by
$G+M$ we denote the graph obtained from $G$ by adding  all the edges of $M$ to $G$.   Every edge in $\overline{G}$  is  called a \dfn{missing edge} of $G$. For a vertex $x\in V(G)$, we will use $N(x)$ to denote the set of vertices in $G$ which are adjacent to $x$.
We define $N[x] = N(x) \cup \{x\}$.  The degree of $x$ is denoted by $d_G(x)$ or
simply $d(x)$.   If  $A, B\subseteq V(G)$ are disjoint, we say that $A$ is \emph{complete} to $B$ if each vertex in $A$ is adjacent to all vertices in $B$, and $A$ is \emph{anticomplete} to $B$ if no vertex in $A$ is adjacent to any vertex in $B$.
If $A=\{a\}$, we simply say $a$ is complete to $B$ or $a$ is anticomplete to $B$.   We use $e(A, B)$ to denote the number of edges between $A$ and $B$ in  $G$. 
The subgraph of $G$ induced by $A$, denoted by $G[A]$, is the graph with vertex set $A$ and edge set $\{xy \in E(G) \mid x, y \in A\}$. We denote by $B \less A$ the set $B - A$,   and $G \less A$ the subgraph of $G$ induced on $V(G) \less A$, respectively. 
If $A = \{a\}$, we simply write $B \less a$    and $G \less a$, respectively.  An \dfn{$(A, B)$-path} in $G$, when $A$ and $B$ are not necessarily disjoint,  is a path with one end in $A$ and  the other in $B$ such that  all its internal vertices lie  in $G\less (A\cup B)$. We simply say an $(a, B)$-path if $A=\{a\}$. It is worth noting that each vertex in $A\cap B$ is an $(A, B)$-path. 
  For a positive integer $k$,  a $k$-vertex is a vertex of degree $k$, and a $k$-clique   is a set of $k$ pairwise adjacent vertices.     
 Let $\mathcal{F}$ be a family of graphs. A graph $G$ is \emph{$\mathcal{F}$-free} if it has no subgraph isomorphic to $H$ for every    $H\in\mathcal{F}$. We simply say $G$ is $H$-free if  $\mathcal{F}=\{H\}$.  The \dfn{join} $G+H$ (resp. \dfn{union} $G\cup H$) of two 
vertex-disjoint graphs
$G$ and $H$ is the graph having vertex set $V(G)\cup V(H)$  and edge set $E(G)
\cup E(H)\cup \{xy\, |\,  x\in V(G),  y\in V(H)\}$ (resp. $E(G)\cup E(H)$).  
We use the convention   ``A :="  to mean that $A$ is defined to be
the right-hand side of the relation.  Finally,  if $H$ is a connected subgraph of  a graph  $G$ and $y\in   V(H)$, we simply say that we   \dfn{contract $H\less y$ onto $y$}  when  we contract  $H$  to a single vertex, that is, contract all the edges of $H$. 
 \medskip

 To prove \cref{t:main}, we need to investigate the basic properties of contraction-critical graphs.  For a positive integer $k$, a graph $G$ is \dfn{$k$-contraction-critical} if $\chi(G)=k$ and every proper minor of $G$ is $(k-1)$-colorable.   
 Lemma~\ref{l:alpha2}   is  a   result 
  of Dirac~\cite{Dirac60}.  

\begin{lem}[Dirac~\cite{Dirac60}]\label{l:alpha2}   Let $G$ be a  $k$-contraction-critical graph. Then for each  $v\in V(G)$, \[\alpha(G[N(v)])\le d(v)-k+2.\] 
\end{lem}

 We will make use of the following results of   Mader~\cite{7con}.   It seems very difficult to improve Theorem~\ref{t:7conn}; it remains open whether every $k$-contraction-critical graph is $8$-connected for all $k\ge8$.  Let $z_1, z_2, z_3, z_4$ be four vertices in a graph $G$.  We say that $G$ contains a \dfn{$K_4$  minor rooted at $z_1, z_2, z_3, z_4$} if there exist $Z_1, Z_2, Z_3, Z_4\subseteq V(G)$ such that $z_i\in Z_i$ and $G[Z_i]$ is connected for  all $i\in [4]$, and $Z_i$ and $Z_j$ are disjoint and there is an edge between $Z_i$ and $Z_j$ in $G$ for $1\le i<j\le4$. 
 
\begin{thm}[Mader~\cite{7con}]~\label{t:rootedK4}   For all $i\in[4]$, let $Z_i\subseteq V(G)$ with $z_i\in Z_i$ such that  $Z_i\cap Z_j=\emptyset$ for   $1\le i<j\le 4$, and $\alpha(G[\{z_1,z_2, z_3, z_4\}]\le2$. If   there exists a $(z_i, z_j)$-path consisting only of vertices from $Z_i\cup Z_j$ for   $1\le i<j\le 4$, then $G$  has a $K_4$ minor rooted at $\{z_1,z_2, z_3, z_4\}$.
\end{thm}

\begin{thm}[Mader~\cite{7con}]\label{t:7conn}  
For all $k \ge 7$, every $k$-contraction-critical graph is $7$-connected.
\end{thm}

  \cref{l:wonderful}  on contraction-critical graphs  turns out to be very powerful, as the existence of pairwise vertex-disjoint paths  is guaranteed without using the connectivity of such  graphs.  The proof of \cref{l:wonderful}  uses  Kempe chains. We recall the proof here as it will be needed in the proof of \cref{l:rootedK4}. Recall that    every edge in $\overline{H}$ is a \dfn{missing edge} of a graph $H$.

   \begin{lem}[Rolek and Song~\cite{RolekSong17a}]\label{l:wonderful} 
Let $G$ be any $k$-contraction-critical graph. Let $x\in V(G)$ be a vertex of
     degree $k + s$ with $\alpha(G[N(x)]) = s + 2$ and let $S \subset N(x)$ with
    $ |S| = s + 2$ be any independent set, where $k \ge 4$ and $s \ge 0$ are integers.
     Let $M$ be a set of missing edges of $G[N(x) \setminus S]$.  Then there
     exists a collection $\{P_{uv}\mid uv\in M\} $ of paths in $G$ such that
     for each $uv\in M$, $P_{uv}$ has ends $u, v$ and all its internal vertices
     in $G \setminus N[x]$. Moreover,  if vertices $u,v,w,z$ with $uv,wz\in M$ are distinct, then
     the paths $P_{uv}$ and $P_{wz}$ are vertex-disjoint.
 \end{lem} 

\begin{proof} Let $G$, $x$, $S$ and $M$ be  given as in the statement.  Let $H$ be obtained from $G$ by contracting $S\cup\{x\}$ to a single vertex, say  $w$.
Then $H$ is $(k-1)$-colorable. Let $c : V(H)\rightarrow \{1,2, \dots, k-1\}$ be a proper  $(k-1)$-coloring of $H$. We may assume that $c(w)=1$.
Then  each of the colors $ 2,  \dots, k-1$ must appear in $G[N(x)\less S]$, else we could assign $x$ the missing color and all vertices in $S$ the color $1$ to  obtain a proper $(k-1)$-coloring of $G$, a contradiction.
Since $|N(x)\less S|=k-2$,   we have $c(u)\ne c(v)$ for any two distinct vertices $u, v$ in $G[N(x)\less S]$.   
 We next claim that  for each $uv\in M$   there must exist a path between $u$ and $v$ with its internal vertices in $G\less N[x]$. Suppose not. 
Let $H^*$ be the subgraph of $H$ induced by the vertices colored  $c(u)$ or $c(v)$ under the coloring $c$. Then $V(H^*)\cap N(x)=\{u,v\}$. Notice that  $u$ and $v$ must belong to different components of $H^*$ as there is no  path between $u$ and $v$ with its internal vertices in $G\less N[x]$.
By switching the colors on the component of $H^*$ containing $u$, we obtain a $(k-1)$-coloring of $H$ with the color $c(u)$ missing on $G[N(x)\less S]$, a contradiction.
This proves that  there must exist a path $P_{uv}$ in $H^*$ with ends $u,v$ and all its internal vertices in $H^*\less N[x]$ for  each $uv\in M$.
Clearly, for any $uv,wz\in M$ with $u,v,w,z$ distinct, the paths $P_{uv}, P_{wz}$ are vertex-disjoint, because  no two  vertices of $u,v,w,z$  are colored the same under the coloring $c$.  
\end{proof}

We shall make frequent use of the following lemma in the proof of \cref{t:main}. 

\begin{lem}\label{l:rootedK4} Let $G$ be any $k$-contraction-critical graph. Let $x\in V(G)$ be a vertex of
     degree $k + s$ with $\alpha(G[N(x)]) = s + 2$ and let $S \subset N(x)$ with
    $ |S| = s + 2$ be any independent set, where $k \ge 4$ and $s \ge 0$ are integers.
  If $M=\{x_1y_1, x_1y_2, x_2y_1, x_2y_2, a_1b_{11}, \dots, a_1b_{1r_1},   \dots, a_mb_{m1}, \dots, a_mb_{mr_m}\}$ is a set of missing edges of $G[N(x)\less S]$, where  the vertices $x_1, x_2, y_1, y_2, a_1, \dots, a_m,
  b_{11}, \dots, b_{mr_m}\in N(x)\less S$ are all distinct, and  for all $1\le i  \le m$,  $a_ib_{i1}, \dots, a_ib_{ir_i}$ are $r_i$ missing edges  with $a_i$ as a common end, and $x_1x_2, y_1y_2\in E(G)$,   then  $G \se G[N[x]]+M$.    \end{lem}
  
  \begin{proof}      By \cref{l:wonderful}, there
     exists a collection $\{P_{uv}\mid uv\in M\} $ of paths in $G$ such that
     for each $uv\in M$, $P_{uv}$ has ends $u, v$ and all its internal vertices
     in $G \setminus N[x]$.   In particular,    for any $1\le i< j \le m$, the paths $P_{a_ib_{i1}}, \dots, P_{a_i b_{ir_i}}$ are vertex-disjoint from the paths $P_{a_j b_{j1}}, \dots, P_{a_j b_{j r_{j}} }$ and $P_{x_1y_1}, P_{x_1y_2}, P_{x_2y_1}, P_{x_2y_2}$; $P_{x_1y_1}$ and  $ P_{x_2y_2}$ are vertex-disjoint; $  P_{x_1y_2}$ and $ P_{x_2y_1} $ are vertex-disjoint but each of $  P_{x_1y_2}$ and $ P_{x_2y_1} $ is not necessarily vertex-disjoint from $P_{x_1y_1}$ and  $ P_{x_2y_2}$.   By contracting each of $P_{a_i b_{i\ell}}\less b_{i\ell}$ onto $a_i$ for all $i\in[m]$ and $\ell\in [r_i]$, we see that $G\se G[N[x]]+(M\less \{x_1y_1, x_1y_2, x_2y_1, x_2y_2\})$.  We next apply \cref{t:rootedK4} to show that  $G\se G[N[x]]+M$. \medskip

     Let $z_1=x_1, z_2=x_2, z_3=y_1, z_4=y_2$.  From the proof of \cref{l:wonderful},  each of $P_{x_1y_1}$, $P_{x_1y_2}$, $P_{x_2y_1}$, $P_{x_2y_2}$ is a Kempe chain, and vertices of $z_1, z_2, z_3, z_4$ are colored differently under the coloring $c$, where $c$ is given in the proof of \cref{l:wonderful}.  For each $i\in[4]$, let  $Z_i$ be the set of vertices $v$  of $G[V(P_{x_1y_1})\cup V(P_{x_1y_2})\cup V(P_{x_2y_1})\cup V(P_{x_2y_2})]$ such that $c(v)=c(z_i)$.      Then  $z_i\in Z_i$ and $Z_i\cap Z_j=\emptyset$ for  $1\le i< j\le 4$.  Note that $\alpha(G[\{z_1, z_2, z_3, z_4\}]\le 2$ because $z_1z_2, z_3z_4\in E(G)$; the $(z_1, z_2)$-path has only one edge $z_1z_2$, the $(z_3, z_4)$-path has only one edge $z_3z_4$, and the $(z_i, z_j)$-path for each $i\in[2]$ and $j\in\{3,4\}$ is $P_{x_iy_j}$.  By \cref{t:rootedK4}, we see that  $G[V(P_{x_1y_1})\cup V(P_{x_1y_2})\cup V(P_{x_2y_1})\cup V(P_{x_2y_2})]$ has a $K_4$ minor rooted at $z_1, z_2, z_3, z_4$. It follows that  $G\se G[N[x]]+M$, as desired.
  \end{proof}

  \begin{rem} \cref{l:rootedK4}  can be   applied      when  \[M=\{x_1y_1, x_1y_2, x_2y_1, x_2y_2, a_1b_{11}, \dots, a_1b_{1r_1},  \dots, a_mb_{m1}, \dots, a_mb_{mr_m}\}\] is a subset of  edges and missing edges of  $G[N(x)\less S]$, where    $x_1, x_2, y_1, y_2, a_1, \dots, a_m,
  b_{11}, \dots, b_{mr_m}\in N(x)\less S$ are all distinct, and $x_1x_2, y_1y_2\in E(G)$.  Under those circumstances, it suffices to apply \cref{l:rootedK4} to $M^*$, where $M^*= \{e\in M\mid e \text{ is a missing edge of } G[N(x)\less S])\}$. It is straightforward to see that  $G\se G[N[x]]+M$. 
  \end{rem}
\begin{figure}[htb]
\centering
\includegraphics[scale=0.6]{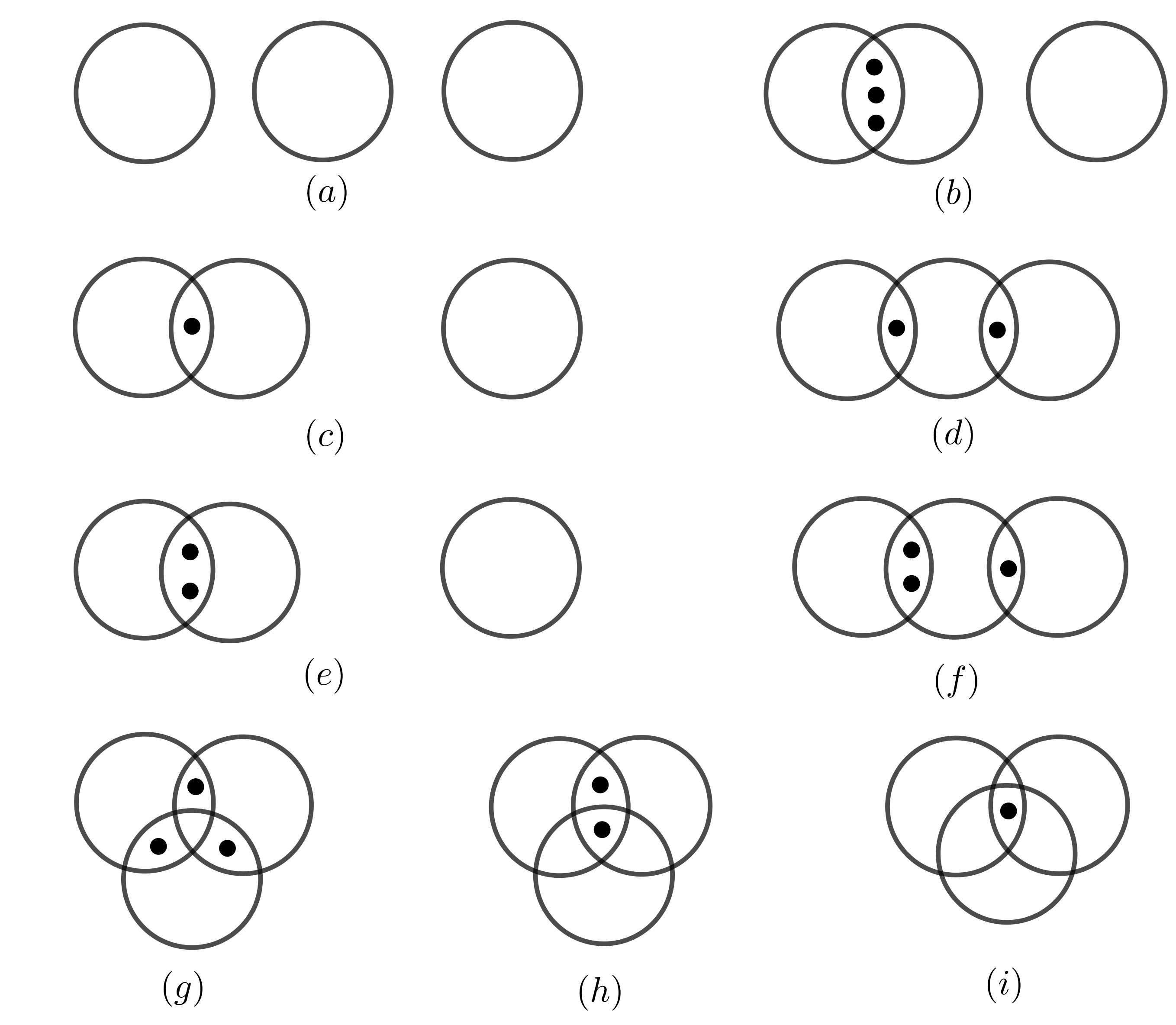}
\caption{The nine possibilities  for three $5$-cliques  in Theorem~\ref{t:goodpaths}.}
\label{fig:threek5s}
\end{figure}

Finally we need a tool to find a desired $\kef$ minor through three different $5$-cliques in   $7$-connected graphs. This method was first introduced  by Robertson, Seymour and Thomas~\cite{RST} to prove Hadwiger's Conjecture for $t=6$: they found a desired $K_6$ minor via three different $4$-cliques in $6$-connected non-apex graphs. This method   was later extended by Kawarabayashi and Toft~\cite{KT05} to find a desired $K_7$ minor via three different $5$-cliques in $7$-connected  graphs.

\begin{thm}[Kawarabayashi and Toft~\cite{KT05}]\label{t:goodpaths}
Let $G$ be a $7$-connected graph such that $|G| \geq 19$.
If $G$ contains  three different  $5$-cliques, say $L_1, L_2, L_3$,   such that  $|L_1\cup L_2\cup L_3|\ge 12$,    that is,  they fit into one of the nine configurations depicted in Figure~\ref{fig:threek5s}, then $G\se K_7$. In particular, $G$ has seven pairwise vertex-disjoint ``good paths", where  a ``good path"  is  an $(L_i, L_j)$-path in $G$ with  $i \neq j$.
\end{thm}

It is worth noting that \cref{t:goodpaths} corresponds to \cite[Lemma 5]{KT05}, where the existence of such seven  ``good paths" follows from the proof  of   \cite[Lemma 5]{KT05}.  \cref{t:goodpaths} implies the following: 

\begin{cor}
\label{c:minorfromgoodpaths}
Let $G$ be a $7$-connected graph such that $|G| \geq 19$.
If $G$ contains  three different  $5$-cliques, say $L_1, L_2, L_3$,   such that  $|L_1\cup L_2\cup L_3|\ge 12$,    that is,  they fit into one of the nine configurations depicted in Figure~\ref{fig:threek5s}, then $G\se\kef$.
Moreover, if $L_1 \cap L_2 \cap L_3 =\emptyset$, then $G\se\ket$.
\end{cor}
\begin{proof} By \cref{t:goodpaths}, $G$ has seven pairwise vertex-disjoint ``good paths". By choosing such seven ``good paths", say $Q_1, \ldots, Q_7$,  with $|V(Q_1)\cup \ldots\cup V(Q_7)|$ minimum, we may assume that no internal vertex of each $Q_i$  lies in $L_1\cup L_2\cup L_3$. Moreover, if $v \in L_i \cap L_j$ for some $i\ne j$, then $v$ does not belong to any ``good path" on at least two vertices. It is then straightforward to verify that each possibility of $L_1, L_2, L_3$, given in Figure~\ref{fig:threek5s}(a-g), together with $Q_1, \ldots, Q_7$, yields a $\ket$ minor; and each possibility of $L_1, L_2, L_3$, given in Figure~\ref{fig:threek5s}(h,i), together with $Q_1, \ldots, Q_7$, yields a $\kef$ minor, as desired.  
\end{proof}
\begin{figure}[htb]
\centering
 \includegraphics[scale=0.6]{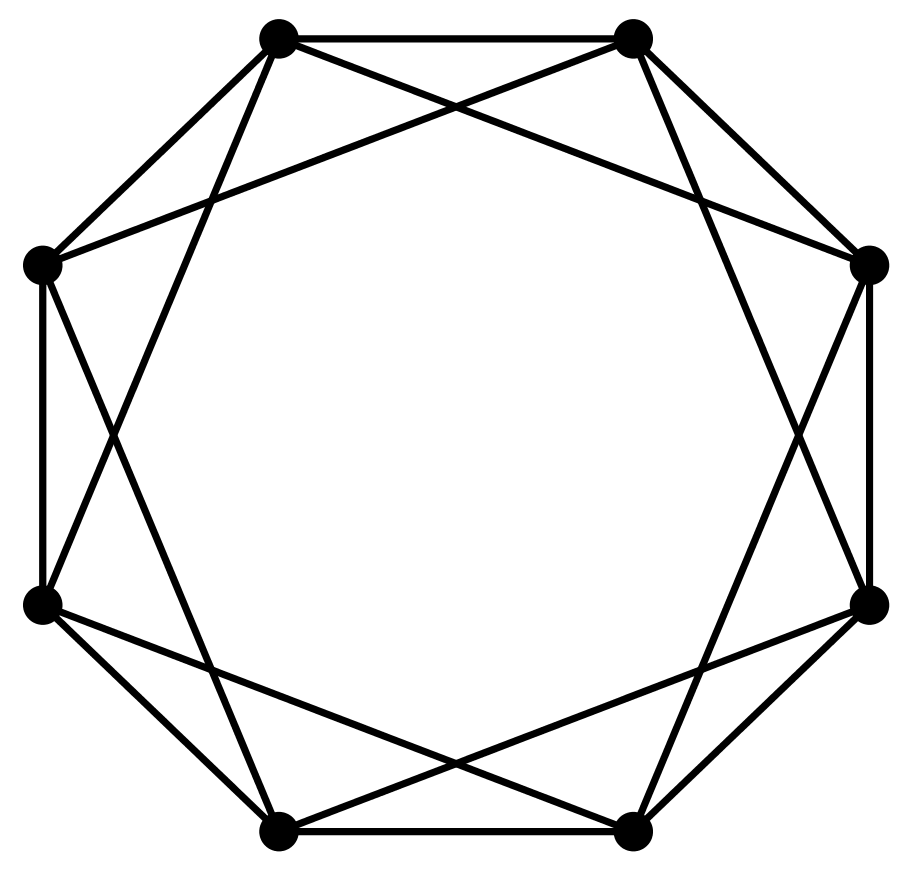}
\caption{The graph $H_8$.}
\label{H8}
\end{figure}

Finally we need a lemma of Rolek, Thomas and the second author. 

\begin{lem}[Rolek, Song and Thomas~\cite{RST22}]\label{l:h8} Let $H$ be a graph with $|H|=8$ and $\alpha(H)=2$. Then $H$  contains   $K_4$  or $H_8$ as a subgraph, where $H_8$  is depicted in Figure~\ref{H8}.  
\end{lem}

\section{Coloring graphs with no $\kef$ minor}\label{s:coloring}

 We begin this section with two lemmas.  \cref{l:noh8} below follows directly from \cref{l:rootedK4}. We give a proof here  without using   \cref{t:rootedK4} (its  proof was written in German). 

\begin{figure}[htb]
\centering
 \includegraphics[scale=0.7]{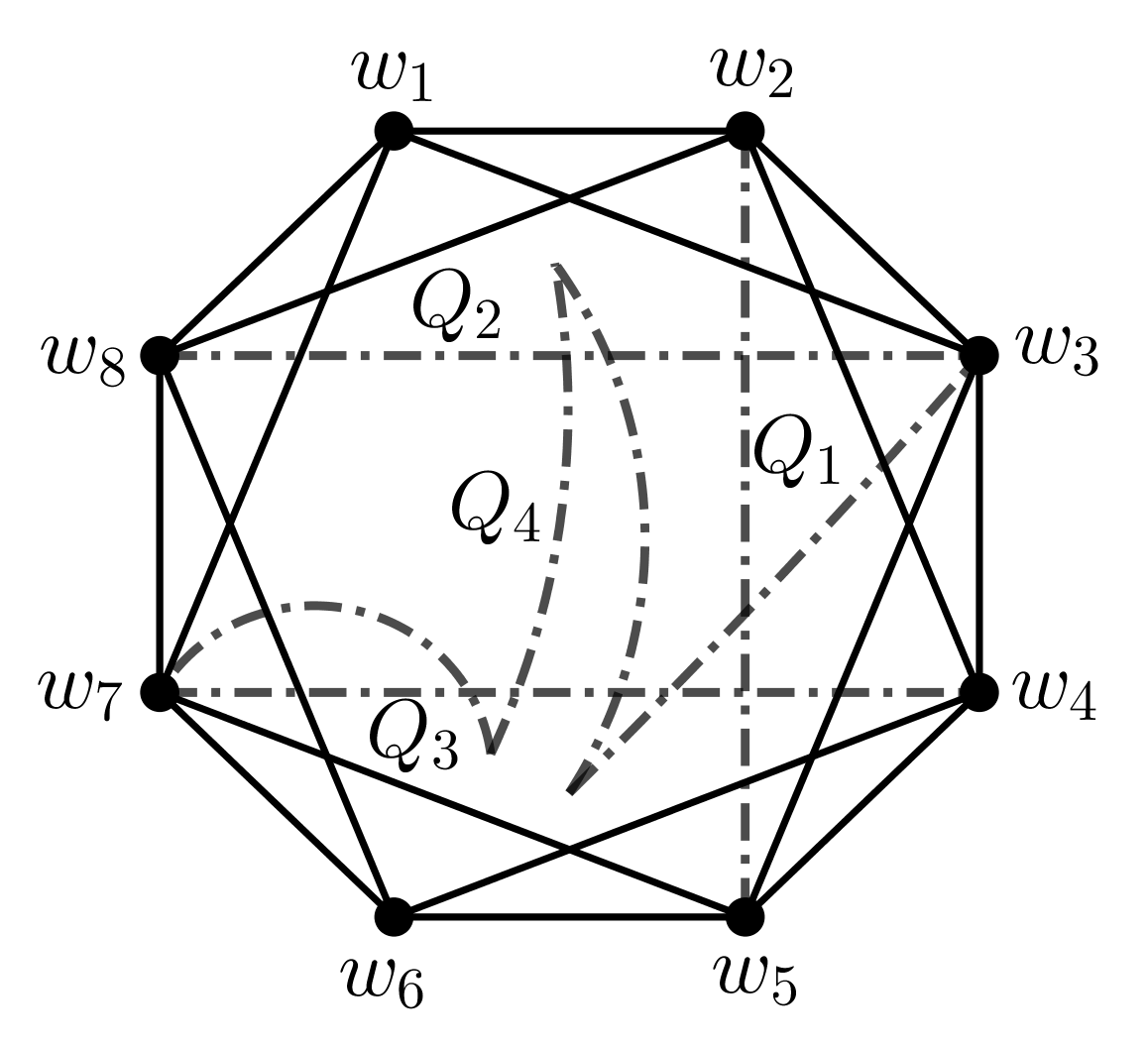}
\caption{Applying \cref{l:wonderful} to   $N(x)$ with $S=\{w_1, w_6\}$ and $M=\{w_2w_5, w_3w_8,  w_4w_7, w_3w_7\}$.}
\label{H8a}
\end{figure}
\begin{lem}
\label{l:noh8}
Let $G$ be an $8$-contraction-critical graph. If there exists $x \in V(G)$ such that $G[N(x)]$ is isomorphic to $H_8$, then $G\se\ket$.
\end{lem}
\begin{proof} Suppose $G[N(x)]=H_8$ for some $x\in V(G)$. Let $w_1, \dotsc, w_8$ be the vertices of $N(x)$ as depicted in 
 Figure~\ref{H8a}. Note that $|H_8|=8$ and $\alpha(H_8)=2$. 
By Lemma~\ref{l:wonderful} applied to $N(x)$ with $S = \{ w_1, w_6 \}$ and $M=\{w_2w_5, w_3w_8, w_4w_7, w_3w_7\}$,  there exist  pairwise vertex-disjoint  paths  $Q_1, Q_2, Q_3$ and another path $ Q_4$ such that   $Q_1, Q_2, Q_3, Q_4$  have ends  in $\{w_2, w_5\}$, $\{w_3, w_8\}$, $\{w_4, w_7\}$ and $ \{w_3, w_7\}$, respectively, and all their internal vertices are in $G\less N[x]$. Note that $Q_4$ and $Q_2$ may have  more than $w_3$   in common, and  $Q_4$ and $Q_3$   may have  more than $w_7$ in common; see   Figure~\ref{H8a}.  Let $u\in V(Q_4)\cap V(Q_2)$   such that the $(w_7, u)$-subpath $Q_4'$ of $Q_4$ and $Q_2$  have exactly $u$ in common; and let $v\in V(Q'_4)\cap V(Q_3)$ such that the $(u, v)$-subpath  $Q_4''$ of $Q_4'$    and $Q_3$ have exactly $v$ in common. Let $Q_4'''$ denote the $(w_7, v)$-subpath of  $Q_4'$. By contracting  $Q_1\less w_5$ onto $w_2$, both $Q_2\less w_8$  and $Q_4''\less v$ onto $w_3$, and both $Q_3\less w_4$ and $Q'''_4 $ onto $w_7$,  we see that $H_8+M$ has a $\kst$ minor after contracting the edge $w_6w_8$. Hence  $G\se G[N[x]]+M\se  \ket$, as desired.
\end{proof}

\begin{lem}
\label{l:nok6k5}
Let $G$ be an $8$-contraction-critical graph.
If $G$ has a $6$-clique $A$ and a $5$-clique $B$ such that $B\not\subseteq A$, then $G\se\ket$.
\end{lem}
\begin{proof} Let $A:=\{a_1, \ldots, a_6\}$ and $B:=\{b_1, \ldots, b_5\}$. Let $t:=|A\cap B|$. Then $0\le t\le 4$ because $B\not\subseteq A$. We may further assume that $b_i=a_i$ for each $i\in[t]$ when $t\ne 0$. Assume first $t=4$. Then $G[A\cup B]$ contains $K_7^<$ as a subgraph. 
By Theorem~\ref{t:7conn}, $G$ is $7$-connected. We obtain a $K_8^<$ minor   by contracting a component of $G\less (A\cup B)$  to a single vertex, as desired. 
 We may assume that  $0\le t\le 3$.  Then there exist $5-t$  pairwise vertex-disjoint paths $Q_{t+1}, \dots, Q_5$  between $A\less B$ and $B\less A$ in $G\less (A\cap B)$. We may assume that $Q_i$ has ends $a_i, b_i$ for each  $i\in\{t+1,  \dots,5\}$.  
Then $G \less \{a_1, \dots, a_5 \}$ is connected, so there must exist a path $Q$ with one end  $a_6$ and the other in  $V(Q_{t+1}\less a_{t+1})  \cup \dots \cup V(Q_5\less  a_5)$, say in $V(Q_5\less  a_5)$. 
We may assume that $Q$  is vertex-disjoint from $Q_{t+1}, \dots, Q_4$ by choosing $Q$ to be a shortest such path.   
Now contracting both $Q_5\less a_5$ and $Q\less a_6$ onto $b_5$, $Q_4\less a_4$ onto $b_4$, and all the edges of $Q_j $   for each $j\in\{t+1, \ldots, 3\}$ when $t\le 2$, we  see that $G\se G[A\cup B]+\{a_{t+1}b_{t+1}, \ldots, a_5b_5, a_6b_5\}\se \ket$.    \end{proof}

We are now ready to prove Theorem~\ref{t:main}, which we restate for convenience.\main* 

\begin{proof} Suppose the assertion is false. Let $G$ be a  graph with no $\kef$ minor such  that $\chi(G) \ge 8$.  
We may choose such a  graph $G$ so that it is $8$-contraction-critical. Then $\delta(G)\ge7$,    $G$ is $7$-connected  by  \cref{t:7conn}, and     $\delta(G) \le 8$ by \cref{t:exfun}. 
Let $x\in V(G)$ be of minimum degree. Since $G$ is $8$-contraction-critical and has no $\kef$ minor, by  \cref{l:alpha2} applied to $G[N(x)]$, we see that   \medskip

\setcounter{counter}{0}

\noindent \refstepcounter{counter}\label{e:alpha} (\arabic{counter})\, $\delta(G)=8$ and $\alpha(G[N(x)]) = 2$.\medskip

\noindent  \refstepcounter{counter}\label{e:74} (\arabic{counter})\, 
 $G $ is $\mathcal{K}_7^{-4}$-free. 
\begin{proof}Suppose $G$ contains an  $H\in \mathcal{K}_7^{-4}$ as a subgraph. Since $G$ is $7$-connected,  we obtain a  
  $\kef$ minor of $G$ by contracting a component of $G\less V(H)$  to a single vertex, a contradiction.\end{proof}

 \noindent  \refstepcounter{counter}\label{e:K4} (\arabic{counter})\, 
$G[N(x)]$ contains a  $4$-clique.

\begin{proof} Suppose $G[N(x)]$ is $K_4$-free. By \cref{l:h8},   $G[N(x)]$ contains $H_8$ as a subgraph. Let the vertices of $H_8$ be  labeled  as in Figure~\ref{H8a}. Then $w_1w_6\notin E(G)$ because $G[N(x)]$ is $K_4$-free.  We consider the worst scenario that $G[N(x)]=H_8$. By \cref{l:noh8}, $G \se \ket$, a contradiction.  
\end{proof}

\noindent \refstepcounter{counter}\label{e:n8} (\arabic{counter})\,  
$n_8\ge 25$, where $n_8$ denotes  the number of vertices of degree eight in $G$. 
 
 \begin{proof} Suppose $|n_8|\le 24$.   Then $e(G)\ge (8n_8+9(|G|-n_8))/2\ge (9|G|-24)/2$. By \cref{t:exfun}, $G\se \kef$, a contradiction. 
 \end{proof}
 
\noindent  \refstepcounter{counter}\label{e:nok6} (\arabic{counter})\,    
$G$ is $K_6$-free. 

\begin{proof} Suppose  $G$ has a $6$-clique $A$. By (\ref{e:n8}), $n_8\ge 25$. Let $y\in V(G)\less A$ be an $8$-vertex in $G$. By  (\ref{e:K4}), $N[y]$ has a $5$-clique $B$  such that  $y\in B$. Then $B\not\subseteq A$. By \cref{l:nok6k5},  $G\se\ket$, a contradiction. \end{proof}

\noindent  \refstepcounter{counter}\label{e:2K4} (\arabic{counter}) \,    
$G[N(x)]$ has two  disjoint $4$-cliques.   

\begin{proof}  Suppose $G[N(x)]$ does not have  two  disjoint $4$-cliques.  By (\ref{e:nok6}),   $G[N(x)]$ is $K_5$-free. Then $\delta(G[N(x)])\ge 3$ because  $\alpha(G[N(x)])=2$.  We claim that   $\delta(G[N(x)])\ge4$.  Suppose not. Let $y\in N(x)$ such that $y$ has exactly three neighbors $y_1, y_2, y_3\in N(x)$. Let $Y:=\{y, y_1, y_2, y_3\}$ and $Z:=\{z_1, z_2, z_3, z_4\}=N(x)\less Y$. Then $Z$ is a $4$-clique in $G$ and $Y$ is not a clique. We may assume that $y_1y_2\notin E(G)$.  Note that $e(y_i, Z)\ge 1$ for each $i\in[2]$. We may assume that $y_1z_1\in E(G)$. Suppose $y_1$ is anticomplete to $\{z_2, z_3, z_4\}$. Then $y_1y_3\in E(G)$ and $y_2$ must  be  complete to $\{z_2, z_3, z_4\}$. By  \cref{l:rootedK4} applied to $N(x)$ with $S=\{y, z_4\}$ and 
 $M=\{y_1z_2, y_1z_3, y_3z_2, y_3z_3, y_2z_1\}$, we obtain a $\ket$ minor from $G[N[x]]+M$ after contracting the edge $yy_2$, a contradiction.  Thus  $e(y_1, Z)\ge2$. Similarly,  $e(y_2, Z)\ge2$.  Note that $e(\{y_1, y_2\}, Z)\le 5$, else $G[Z\cup\{x, y_1, y_2\}]$ is not $\kst$-free, contrary to (\ref{e:74}). We may assume that $e(y_1, Z)=2$ and  $y_1z_2\in E(G)$. Then $y_1$ is anticomplete to $\{ z_3, z_4\}$ and so $y_2$ must  be  complete to $\{z_3, z_4\}$. By   \cref{l:rootedK4} applied to $N(x)$ with $S=\{y, z_1\}$ and 
 $M=\{y_1z_3, y_1z_4, y_3z_3, y_3z_4, y_2z_2\}$, we obtain a $\kef$ minor from $G[N[x]]+M$ after contracting the edge $yy_1$, a contradiction. This proves that $\delta(G[N(x)])\ge 4$, as claimed.\medskip

By (\ref{e:K4}), $G[N(x)]$ has a $4$-clique $A$.  Let  $N(x) = \{ a_1, \dotsc, a_4, b_1, \dotsc, b_4 \}$, where $A=\{ a_1, \dotsc, a_4 \}$ is a $4$-clique and $B:=\{ b_1, \dotsc, b_4 \}$ is not. 
We may assume that $b_1 b_2 \notin E(G)$.  Since $G[N(x)]$ is $K_5$-free and $\alpha(G[N(x)])=2$, we   may further assume that $b_1a_1, b_2a_2\notin E(G)$. Then $b_1a_2, b_2a_1\in E(G)$. Note that $e(b_i, A)\ge 2$ for each $i\in[2]$ because $\delta(G[N(x)])\ge 4$. On the other hand,   $e(\{b_1, b_2\}, A)\le 5$, else $G[A\cup\{x, b_1, b_2\}]$ is not $\kst$-free, contrary to (\ref{e:74}). We may assume that $e(b_1, A)=2$ and  $2\le e(b_2, A)\le3$. We may further assume that $b_1a_3, b_2b_3\in E(G)$. Then $b_1  a_4\notin E(G)$. It follows that  $b_2a_4\in E(G)$ and $b_1$ is complete to $\{b_3, b_4\}$.  By  \cref{l:rootedK4} applied to $N(x)$ with $S=\{b_1, a_4\}$ and 
 $M=\{b_2a_2, b_2a_3, b_3a_2, b_3a_3, b_4a_1\}$, we obtain a $\kef$ minor from $G[N[x]]+M$ after contracting the edge $b_1b_4$, a contradiction.   
\end{proof}

To complete the proof, recall that $x$ is an $8$-vertex in $G$. Since $n_8\ge25$ by (\ref{e:n8}), let   $y\in V(G)\less N[x]$   such that $y$ is an $8$-vertex in $G$. By (\ref{e:2K4}),  let $A, B$ be disjoint $4$-cliques of $G[N(x)]$ and let $C$ be a $4$-clique of $G[N(y)]$. Let $L_1:=A\cup\{x\}$, $L_2:=B\cup\{x\}$ and $L_3:=C\cup \{y\}$. Then $L_1\cap L_2=\{x\}$ and $L_1\cap L_2\cap L_3=\emptyset$.   By Corollary~\ref{c:minorfromgoodpaths},   $L_1, L_2, L_3$ are not as depicted in Figure~\ref{fig:threek5s}(c,d,f,g) because  $G$ has no $\kef$ minor.   We may   assume that  $|L_1 \cap L_3| \ge |L_2 \cap L_3|$. Then  $|L_1 \cap L_3|   \ge2$, else $L_1, L_2, L_3$ are as depicted in Figure~\ref{fig:threek5s}(c,d,g).  \medskip

 Let $L_1:=\{x,  x_1, \ldots,  x_4\}$, $L_2:=\{x, y_1, \ldots,  y_4\}$ and $L_3:=\{y, z_1, \ldots z_4\}$.  
Suppose first    $|L_1 \cap L_3| = 4$. We may assume that $x_i=z_i$ for each $i\in[4]$.  
Note that  $G[L_1 \cup L_3]=K_6^-$.
By Menger's Theorem,  there exist four pairwise vertex-disjoint   $(\{z_1, \ldots, z_4\}, \{y_1, \ldots, y_4\})$-paths, say $Q_1, \ldots, Q_4$,  in $G\less\{x, y\}$. We may assume that $Q_j$ has ends $y_j, z_j$ for each $j\in[4]$. Let $Q$ be a shortest path in $G\less \{x, y,  y_1, y_2, y_3, z_4\}$ with one end $y_4$ and the other in  $ V(Q_1\less y_1)\cup V(Q_2\less y_2)\cup V(Q_3\less y_3)$. We may assume that $Q$ is vertex-disjoint from $Q_1$ and $Q_2$ but contains a vertex on $Q_3\less y_3$.  Let $u\in V(Q)\cap V(Q_3)$ and $v\in V(Q)\cap V(Q_4)$ such that  the $(u,v)$-subpath  of $Q$ has no internal vertex in $V(Q_3)\cup V(Q_4)$. Let  $Q_3^*$ be  the $(u, y_3)$-subpath of $Q_3$ and $Q_4^*$ be the $(v, z_4 )$-subpath of $Q_4$.   Then  $G\se G[L_1\cup L_2\cup L_3]+\{y_1z_1, y_2z_2, y_3z_3, y_4z_3,y_4z_4\}\se \kef$ by first contracting all the edges of $Q_1, Q_2$, then  the $(y_4, v)$-subpath  of $Q_4$ onto $y_4$, $Q_4^*\less v$ onto $z_4$, $Q_3^*\less u$ onto $y_3$, and both $Q\less v$ and the $(u, z_3)$-subpath of $Q_3 $ onto $z_3$, a contradiction.  Suppose next  $|L_1 \cap L_3|= 3$. Then $G[L_1\cup L_3]$ is not $\mathcal{K}_7^{-4}$-free, contrary to (\ref{e:74}). Thus  $|L_1 \cap L_3| \le 2$.  Since $|L_1 \cap L_3|   \ge2$ and  $|L_1 \cap L_3| \ge |L_2 \cap L_3|$,  we see that   $|L_1 \cap L_3|=2$ and $0\le  |L_2 \cap L_3| \le2$.  Note that $ |L_2 \cap L_3|\ne 0$, else $L_1, L_2, L_3$
 are as depicted in Figure~\ref{fig:threek5s}(f).  Thus $|L_1 \cap L_3|=2$ and $1\le  |L_2 \cap L_3| \le2$.  We may  assume that $x_3=z_1, x_4=z_2$.  
\medskip

We first consider the case $  |L_2 \cap L_3| =2$. We may assume that $y_3=z_3, y_4=z_4$. Then there exist two vertex-disjoint $(\{x_1, x_2\}, \{y_1, y_2\})$-paths, say $Q_1, Q_2$, in $G\less\{x, x_3, x_4, y_3, y_4\}$. We may assume that $Q_j$ has ends $x_j, y_j$ for each $j\in[2]$. We may further assume   $y\notin V(Q_1)$.  Let $Q_2^*$ be the $(y_2, y)$-subpath of $Q_2$ when $y\in V(Q_2)$. 
Then $G\se G[L_1\cup L_2\cup L_3]+\{x_1y_1, x_2y_2\}\se \ket$ by contracting all the edges of $Q_1$ and $Q_2$ when $y\notin V(Q_2)$; and $G\se G[L_1\cup L_2\cup L_3]+\{x_1y_1, x_2y, y_2y\}\se   \mathcal{K}_8^{-2}$   by first contracting all the edges of $Q_1$, then $Q_2^*\less y$ onto $y_2$, and the $(y, x_2)$-subpath of $Q_2$ onto $x_2$ when $y\in V(Q_2)$, a contradiction. \medskip

It remains to consider the case $  |L_2 \cap L_3| =1$. We may assume that $z_3=y_4$. Then there exist three vertex-disjoint $(\{x_1, x_2, z_4, y\}, \{y_1, y_2, y_3\})$-paths, say $Q_1, Q_2, Q_3$, in $G\less\{x, x_3, x_4,  y_4\}$. We may assume that $Q_1$ has ends $\{x_1, y_1\}$. Note that $L_3\less (L_1\cup L_2)=\{y, z_4\}$. By symmetry, we may further assume that   $Q_3$ has ends $y_3, z_4$.    Then $Q_2$ has ends  $y_2, x_2$,  or   $y_2, y$. By contracting all the edges of $Q_1, Q_2$ and $Q_3$, we see that  $G\se G[L_1\cup L_2\cup L_3]+\{x_1y_1, x_2y_2, y_3z_4\}\se  \ket$   in  the former  case, and 
  $G\se G[L_1\cup L_2\cup L_3]+\{x_1y_1, y_2y, y_3z_4\}\se  \ket$     in the latter case, a contradiction.   \medskip

This   completes the proof of  \cref{t:main}.  
 \end{proof} 

\section{An extremal function for $\kef$ minors}\label{s:exfun}

 Throughout this section, if  $G$ is a graph and $K$ is a subgraph of $G$, then by $N(K)$ we denote
the set of vertices of $V(G)\less V(K)$ that are adjacent to a vertex of $K$.
 If $V(K)=\{x\}$, then  $N(K)=N(x)$. 
It can be easily checked that for each vertex  $x\in V(G)$, if $K$ is a component of $G\less N[x]$, then $N(K)$ is a minimal separating set of $G$. \medskip

We first  give a brief outline of the proof of \cref{t:exfun}. We follow the main ideas in   \cite{SongThomas2006}. Suppose for a contradiction that $G$ is a counterexample to \cref{t:exfun} with  as few vertices as possible.   Since deletion or contraction of edges does not produce smaller counterexamples,
it follows easily that $G$ has minimum degree at least five, and   every edge of $G$ belongs to   at least four triangles. With some effort it can be shown that $G$ is $5$-connected, and  has at most one $5$-vertex  but  no $6$-vertex and no $7$-vertex. As $e(G)=4.5|G|-12$, we see that $G$ has an  $8$-vertex. Fix such a vertex $x$.  We then show that $x$ is not adjacent to a $5$-vertex in $G$. If $G\less N[x]$ has a 
 a component $K$  such that 
$ M \subseteq N(K)$, where $M$ is the set of all vertices
of $N(x)$ that are not adjacent to every other vertex of $N(x)$, then we can find a vertex  $y\in N(x)$ such that $G[N(x)\less y]\se \ksix$, and so $G\se\kef$
  by contracting the connected graph $G[V(K)\cup \{y\}]$  to a single vertex. Thus we may assume that for no $8$-vertex $x$ 
such a component
exists. In particular, $G\less N[x]$ is disconnected. In the next step we  prove  that 
there is no component $K$ of $G\less N[x]$ with $|K|\ge2$ such that 
$d_G(v)\ge9$ for all
vertices $v\in V(K)$, except possibly one. 
  In the last step, we
select an $8$-vertex
$x\in V(G)$   to minimize the size of a
component $K$ of $G\less N[x]$ with $|K|\ge2$.
It follows easily that $K$ does not have a vertex that is an $8$-vertex in $G$.\medskip

We next prove  two  lemmas that will be needed in the proof of Theorem~\ref{t:exfun}.

\begin{lem}\label{l:computer}
Let $H$ be a graph with $|H| = 8$ and $\delta(H) \geq 4$. Then there exists  $x\in V(H)$ such that $H\less x$ has a  $\ksix$ minor.  
\end{lem}

\begin{proof} We may assume that  $\delta(H) =4$ and every edge   is incident with a $4$-vertex in $H$. Suppose  $H\less x$ has no $\ksix$ minor for every $x$ in $H$. 
Suppose first  that $H$ has a $4$-clique, say $A:=\{a_1, \ldots, a_4\}$. Let $B:=V(H)\less A $. By the minimality of $e(H)$, we may assume that $a_i$ is a $4$-vertex in $H$ for each $i\in[3]$.  Let $b_i\in B$ be the unique neighbor of $a_i$ in $B$ for each $i\in[3]$. Note that $b_1, b_2, b_3$ are not necessarily distinct.  We claim that for all $a\in A$ and $b\in B$, if $ab\in E(H)$, then $a$ and $b$ have at least one common neighbor. Suppose not.  We may assume that $a\ne a_1$.  Let $H^*:=H\less a_1$. Then    $e(H^*/ab)=e(H^*)-1=(e(H)-4)-1\ge (16-4)-1=11=e(K_6)-4$. Thus $H\less a_1$     has a  $\ksix$ minor, a contradiction. This implies that  $b_ia_4\in E(H)$ for each $i\in[3]$. Note that     $e(b, A)\ge 1$ for each $b\in B$.  It follows that  $a_4$ is complete to $B$ and so $d(a_4)=7$. By the minimality of $e(H)$, we see that each vertex in $B$ is a $4$-vertex in $H$, which is impossible. This proves that $H$ is $K_4$-free. 
Suppose next  $\alpha(H)=2$. By  \cref{l:h8}, $H$ contains $H_8$ as a subgraph. Let the vertices of $H$ be   labeled as in Figure~\ref{H8}.   It is simple to check that   $H_8\less w_2$ has a  $\ksix$ minor after contracting the edge $w_1w_3$.  Thus  $H$ is $K_4$-free and   $\alpha(H)\ge 3$.  \medskip

It is straightforward to check that $H=K_{4,4}$ when $\alpha(H)=4$; and $K_{4,4}\less x$ has  a  $\ksix$ minor for every $x$ in $K_{4,4}$.  Thus  $\alpha(H)=3$. Let $S:=\{x_1, x_2, x_3\}$ be an independent set  of $H$ and let $V(H)\less S:=\{y_1, \ldots, y_5\}$. Since $d(x_i)\ge 4$ for each $i\in[3]$, we may assume that $\{y_1, y_2\}$ is complete to $S$.  Suppose  $y_3$ is complete to $S$. Since $K_{3, 4}$ has a $\ksix$ minor, we see that  $y_4$ is  neither complete to $S$ nor complete to $\{y_1, y_2, y_3\}$.  We may assume that $y_4x_1\notin E(H)$  and $y_4y_1\notin E(H)$. Then $H\less y_5$ has a $\ksix$  minor after contracting the edge $x_1y_1$, a contradiction. This proves that no $y_j$ is complete to $S$ for each $j\in\{3,4,5\}$.  It follows that $d(x_i)=4$, and $e(y_j, S)=2$ for each $i\in[3]$ and $j\in\{3,4,5\}$; in addition, we may assume that  $y_3$ is complete to $\{x_1, x_2\}$,  $y_4$ is complete to $\{ x_2, x_3\}$ and $y_5$ is complete to $\{ x_1, x_3\}$.   If $H[ \{y_3, y_4, y_5\}]=K_3$, then $H\less y_2$ has a $\ksix$  minor after contracting the edge $x_1y_1$, a contradiction. Thus we may assume that $y_3y_5\notin E(G)$ and $y_3y_2\in E(G)$. Then either $y_3y_4\in E(G)$ or $y_3y_1\in E(G)$.  Thus  $H\less y_5$ has a $\ksix$  minor after contracting the edge $x_3y_4$,  a contradiction.   
\end{proof}

Lemma~\ref{l:k4-} follows from the proof of Lemma~16 of J\o rgensen~\cite{Jor01}. We recall the  proof here   for convenience. 

\begin{lem}[J\o rgensen~\cite{Jor01}]\label{l:k4-} Let $G$ be a $4$-connected graph and let $S\subseteq V(G)$  be a separating set of four vertices. Let $G_1$ and $G_2$ be proper subgraphs of G so that $G_1\cup G_2=G$ and $G_1\cap G_2=G[S]$.   Let $d_1$    be the  maximum number of edges that can be added to $G_2$ by contracting edges of $G$ with at least one 
end in $G_1$.  If $|G_1|\ge 6$, then   \[e(G[S]) + d_1 \geq 5.\]
\end{lem}

\begin{proof}  
Let $x \in V(G_1) \setminus S$.
Then there exist four pairwise internally vertex-disjoint $(x, S)$-paths, say  $Q_1, \ldots, Q_4$,  in $G_1$.
For each $i \in [4]$, let $s_i$ be the vertex in $V(Q_i) \cap S$.
If all four of these paths have length one, then, since $|G_1| \geq 6$, we may choose a vertex $y \in V(G_1) \setminus (S \cup \{ x \})$.
Then there are at least three internally vertex-disjoint  $(y, S)$-paths in $G_1\setminus \{ x \}$.
Contracting some of these paths results in $S$ having at least five edges, as desired. \medskip

We may now assume that $Q_1$ has length at least two.
Since $\{ x, s_1 \}$ is not a separating set, there is a path $Q$ from a vertex on $Q_1 \setminus \{ x, s_1 \}$ to a vertex on $Q_j \setminus x $ for some $j\in \{ 2, 3, 4 \}$, so that only the ends  of $Q$ belong to $V(Q_1) \cup V(Q_2) \cup V(Q_3) \cup V(Q_4)$.
We may assume that  $j = 2$.
Since $\{ x, s_1, s_2 \}$ does not separate the graph, there is a path $R$ from $V(Q_1\less\{x, s_1\}) \cup V(Q_2\less s_2) \cup V(Q) $ to $Q_\ell\setminus  x$ for some  $\ell\in\{3, 4\}$, so that only the ends of $R$ belong to $V(Q_1) \cup V(Q_2) \cup V(Q_3) \cup V(Q_4) \cup V(Q)$.
The result follows from the existence of these paths.
\end{proof}

We are now ready to prove Theorem~\ref{t:exfun}, which we restate for convenience.\exfun*

 \begin{proof} Suppose the assertion is false. Let $G$ be a graph on   $n \geq 8$ vertices with  $ e(G) \geq 4.5n-12$ and, subject to this, $n$ is minimum. We may assume that $e(G) = \left\lceil 4.5n-12 \right\rceil$.
It is simple to check that $G\se  \kef$ for each  $n \in\{ 8,9\}$. Thus  $n \geq 10$.  We next prove several claims. \medskip

\setcounter{counter}{0}
\noindent {\bf Claim\refstepcounter{counter} \label{c:d5} \arabic{counter}.}  
 $\delta(G) \geq 5$.
 
\begin{proof} Suppose $\delta(G)\le 4$.  Let $x\in V(G)$  with $d(x)\le 4$. Then 
	\[e(G \setminus x)  =e(G)-d(x)\ge \lceil 4.5n - 12 \rceil - 4  >  \lceil 4.5|G \setminus x| - 12 \rceil.\] 
	Thus  $G \setminus x $ has a $ \kef$ minor by the minimality of $G$,   a contradiction.
\end{proof}

\noindent {\bf Claim\refstepcounter{counter}\label{c:triangles} \arabic{counter}.}  
Every edge in $G$ belongs to at least $4$ triangles. Moreover, if $x\in V(G)$ is a $5$-vertex, then $G[N[x]]=K_6$. 
 
\begin{proof}
 Suppose there exists    $e \in E(G)$ such that  $e$ belongs to at most three  triangles. Then
	\[e(G / e) \ge  \lceil 4.5n - 12 \rceil - 4  > \lceil 4.5|G / e| - 12 \rceil.\]
	Thus $G /e \se \kef$ by the minimality of $G$, a contradiction. This  implies that $G[N[x]]=K_6$ for each $5$-vertex $x$  in $G$. 
\end{proof}

\noindent {\bf Claim\refstepcounter{counter}\label{c:55} \arabic{counter}.}  
No two $5$-vertices  in $G$ are adjacent.
 
\begin{proof} 
Suppose  there exist two distinct $5$ vertices, say $x, y$, in $G$ such that $xy \in E(G)$. 
Then $|G \setminus \{ x, y \}|=n-2 \geq 8$     and  \[ e(G \setminus \{ x, y \}) = e(G)-9=    \lceil 4.5(n-2) - 12 \rceil.\] 
 Thus $G \setminus \{ x, y \}$ has a  $\kef$ by the minimality of $G$, a contradiction.  
 \end{proof}

Let $S$ be a minimal separating set of vertices in $G$, and let $G_1$ and $G_2$ be proper subgraphs of $G$ so 
that $G=G_1\cup G_2$ and $G_1\cap G_2=G[S]$. 
For each $i\in[2]$, let $d_i$ be the  maximum number of edges that can be added to $G_{3-i}$ by contracting edges of $G$ with at least one 
end in $G_i$. More precisely, let $d_i$ be the 
largest integer so that $G_i$ contains pairwise disjoint sets of vertices 
$V_1,   \dots, V_p$ so that $G_i[V_j]$ is connected, 
 $|S\cap V_j|=1$ for  $1\le j\le p :=|S|$, and so that the graph obtained from $G_i$ by contracting each of $G_i[V_1],   \dots, G_i[V_p]$ to a single vertex 
and deleting $V(G)\less \bigcup_{j=1}^p V_j$ has 
$e(G[S])+d_i$ edges.   It follows from the minimality of $G$ that for each $i\in[2]$, 
\[ e(G_i) + d_{3-i} < 4.5|G_1| - 12 \,\,  \text{ if } |G_i|\ge8.\tag{$*$}\]    
 
\noindent {\bf Claim\refstepcounter{counter}\label{c:nodis7} \arabic{counter}.}   
If $|G_i|=7$ for some $i\in [2]$, then  $|S|\le 5$,    $G_i$ contains $K_7^-$ or $K_7^=$ as a spanning subgraph, and each  missing edge of $G_i$  lies in  $G[S]$.  
\begin{proof}
Suppose, say,  $|G_1| = 7$.    Let $C$ be a component of $G_2 \setminus S$.  We first prove that  $G_1\less S$ is  connected. Suppose not. By Claims~\ref{c:d5}, \ref{c:triangles} and  \ref{c:55},    $G_1\less S$ must contain two nonadjacent $5$-vertices in $G$ with   $G[S]=K_5$. Thus      $G_1=K_7^-$ and so  $G\se \ket$ by contracting   $C$   to  a single vertex, a contradiction.   Thus    $G_1\less S$ is  connected, and so $G_1\less S$ has least one $6$-vertex, say $x$, in $G$.  
Suppose next $|S|=6$. Then  $G[S]=G[N(x)]$ contains $K_{2,2,2}$ as a spanning subgraph by Claim~\ref{c:triangles}. But then  $G\se \kef$ by contracting   $C$   to a single vertex, a contradiction.  This proves that $|S|\le 5$ and so $|G_1\less S|\ge 2$.  \medskip

 We next show that  no vertex in $G_1\less S$ is a $5$-vertex in $G$. Suppose not.  Let $  y\in V(G_1\less S)$  be a $5$-vertex in $G$.  Then $xy\in E(G)$ and $G_1= K_7^-$ by Claim~\ref{c:triangles}. Note that   $G[S]$ is a complete subgraph because $G[N(y)]=K_5$  and $G[N(x)]$ contains $K_{2,2,2}$ as a subgraph by Claim~\ref{c:triangles}.   Then $|S|\le 3$, else $G\se\kef$ by contracting $C$  to a single vertex. Suppose $|G_2|\ge 8$. Then 
\begin{align*}
e(G_2)&=e(G)-e(G_1)+e(G[S])\\
&=\lceil 4.5n-12\rceil-20+{{|S|}\choose2}\\
&= \left\lceil\big(4.5\times (n-(7-|S|))-12\big) +  4.5\times (7-|S|)\right\rceil -20+{{|S|}\choose2}\\
&\ge \lceil 4.5|G_2|-12\rceil
\end{align*}
because $|S|\le 3$. By the minimality of $G$, we see that $G_2\se\kef$, a contradiction. Thus $|G_2|\le 7$.   If $|G_2|=6$, then $G_2=K_6$ and   $n=7+6-|S|=13-|S|$; but then  \[\lceil 4.5\times (13-|S|)-12\rceil= e(G)=e(G_1)+e(G_2)-e(G[S])= 20 + 15-{{|S|}\choose2},\] 
which is impossible because $|S|\le3$. Thus $|G_2|=7$ and so  $n=7+7-|S|=14-|S|$.  Using a similar argument for $G_1$, we see that $G_2\less S$  is connected and so $G_2=K_7$ or $K_7^-$. Note that $G_2=K_7^-$ when $|S|=3$, else $G\se\kef$ by contracting $G_1\less S$  to a single vertex. But then 
 \[\lceil 4.5\times (14-|S|)-12\rceil= e(G)=e(G_1)+e(G_2)-e(G[S])\le  20 + (21-\max\{0, |S|-2\})-{{|S|}\choose2},\] 
which is impossible because $|S|\le3$.  This proves that no vertex in $G_1\less S$ is a $5$-vertex in $G$. It then follows that each vertex in $G_1\less S$ is a $6$-vertex in $G$. By Claim~\ref{c:triangles} and the fact that $|S|\le 5$, we see  that $G[S]$ is isomorphic to $K_{|S|}$, $K_{|S|}^-$, or $K_{|S|}^=$,  and every vertex in $V(G_1\less S)$ is adjacent to all the other vertices in $G_1$. Thus $G_1$ contains $K_7^-$ or $K_7^=$ as a spanning subgraph such that  all its  missing edges lie in $G[S]$. 
 \end{proof}

\noindent {\bf Claim\refstepcounter{counter}\label{c:no7} \arabic{counter}.}  
Neither $G_1$ nor $G_2$ has exactly seven vertices.
 
\begin{proof}
Suppose not, say $|G_1| = 7$.  By Claim~\ref{c:nodis7},   $|S|\le 5$, $G_1$ contains $K_7^-$ or $K_7^=$ as a spanning subgraph, and each  missing edge of $G_1$ lies in  $G[S]$.  
We next prove that $|G_2|\ge 8$. Suppose $  |G_2|\le 7$. Note that $|G_2|\ge6$ by Claim~\ref{c:d5}. Suppose $|G_2|=6$. Then  $G_2=K_6$ and $G[S]=K_5$. But then $n=|G_1|+|G_2|-| S|=7+6-5<10$, a contradiction. Thus $|G_2|=7$ and  so $n= 14-|S|$.   Since $n\ge10$, we see that $|S|\le4$. 
By Claim~\ref{c:nodis7}, $G_2$ contains $K_7^-$ or $K_7^=$ as a spanning subgraph, and each  missing edge of $G_2$ lies  in  $G[S]$.    Suppose $  |S|=4$. Then $G[S]=K_4^=$, else $G\se \kef$ by contracting $G_2\less S$  to a single vertex.   Let $y \in V(G_2)\less S$ and $z\in S$ such that $z$ is incident with a  missing edge in $G[S]$. It is simple to check that  $G \se \kef$  by contracting the edge $yz$ and $G_2\less (S\cup\{y\})$  to a single vertex, a contradiction.    Thus $|S|\le3$.  As noted in the proof of Claim~\ref{c:nodis7},  we see that $G_1=G_2=K_7^-$ when $|S|=3$.   But then   
  \[\lceil 4.5\times(14-|S|)-12\rceil= e(G)=e(G_1)+e(G_2)-e(G[S])\le 2\times (21-\max\{0, |S|-2\}) -{{|S|}\choose2},\] 
which is impossible because $|S|\le 3$. This proves that $|G_2|\ge 8$. \medskip

 Recall that $G_1$ contains $K_7^-$ or $K_7^=$ as a spanning subgraph, and each  missing edge of $G_1$ lies in  $G[S]$.  It follows that  $e(G[S])+d_1={{|S|}\choose2}$. Suppose $G_1=K_7^=$. Then $4\le |S|\le 5$. But then \begin{align*}
e(G_2)+d_1&=e(G)-e(G_1)+e(G[S])+d_1\\
&=\lceil 4.5n-12\rceil-19+ {{|S|}\choose2} \\
&= \left\lceil\big(4.5\times (n-(7-|S|))-12\big) +  4.5\times (7-|S|)\right\rceil -19+{{|S|}\choose2}\\
&\ge \lceil 4.5|G_2|-12\rceil, 
\end{align*}
 contrary to ($*$) because $4\le |S|\le 5$ and $|G_2|\ge 8$.  
 Thus  $G_1=K_7^-$ or $K_7$. Note that  $|S|\le 3$, else $G\se \kef$ by contracting  a component of $G_2\less S$  to a single vertex.   But then 
\begin{align*}
e(G_2)+d_1&=e(G)-e(G_1)+e(G[S])+d_1\\
&\ge \lceil 4.5n-12\rceil-21+ {{|S|}\choose2} \\
&= \left\lceil\big(4.5\times (n-(7-|S|)\big)-12)+4.5\times (7-|S|)\right\rceil-21+{{|S|}\choose2}\\
&\ge \lceil 4.5|G_2|-12\rceil, 
\end{align*}
 contrary to ($*$) because $|S|\le 3$ and $|G_2|\ge 8$.  This proves Claim~\ref{c:no7}. 
\end{proof}

Observe that, if  $|G_1|\ge 8$ and $|G_2|\ge 8$, then by ($*$), we have 
\begin{align*}
4.5n - 12\le e(G) &= e(G_1) + e(G_2) - e(G[S])  \\
&<(4.5|G_1|-12-d_2)+(4.5|G_2|-12-d_1)- e(G[S]) \\
&=4.5(n + |S|) - 24 - d_1 - d_2 - e(G[S]).
 \end{align*}
It follows that 
\[ 9|S| > 24 + 2\big(d_1 + d_2 + e(G[S])\big) \, \, \text{ if } |G_1|\ge 8\,  \text{ and } |G_2|\ge 8. \tag{$**$} \]\medskip

\noindent {\bf Claim\refstepcounter{counter}\label{c:5conn}  \arabic{counter}.}  
$G$ is $5$-connected.
 
\begin{proof} Suppose $G$ is not 5-connected. Let $S$ be a minimal separating set of $G$, and  $G_1, G_2,   d_1, d_2$ be defined as above. 
Then $|G_1|\ne 6$, else by Claim~\ref{c:triangles},   we have $G_1=K_6$ and so $G[S]=K_5$, a contradiction. Similarly, $|G_2|\ne 6$. By Claim~\ref{c:no7}, $|G_1|\ge 8$ and $|G_2|\ge 8$. By ($**)$, 
  $|S| \geq 3$.
If $|S| = 3$, then either $e(G[S]) \geq 2$ or $\min\{ d_1, d_2 \} \geq 1$; in either case, $d_1 + d_2 + e(G[S]) \geq 2$, contrary to $(**)$.
Therefore, $|S| =4$ and $G$ is $4$-connected. By Lemma~\ref{l:k4-}, $e(G[S])+d_1\ge5$. Note that $d_2\ge 1$ when $S$ is not a $4$-clique, and $e(G[S])=6$ when  $S$ is  a $4$-clique. In either case, we have $d_1 + d_2 + e(G[S])\ge6$, contrary to $(**)$. \end{proof}

\noindent {\bf Claim\refstepcounter{counter}\label{c:almostclique}   \arabic{counter}.}  
  If there exists $x \in S$ such that $S \setminus x$ is a clique, then $G[S] = K_5$.
 
\begin{proof} Suppose  $S \less x$ is  a clique but $G[S]\ne K_5$.   Let $G_1$ and $G_2$ be as above. Then $|G_1|\ne 6$, else   $G_1=K_6$ and $G[S]=K_5$. Similarly, $|G_2|\ne 6$. By Claim~\ref{c:no7}, $|G_1|\ge 8$ and $|G_2|\ge 8$.  By Claim~\ref{c:5conn}, $|S| \geq 5$. 
If $S$ contains a $6$-clique, then $G\se K_8^-$  by contracting a component of $G_1 \setminus S$  and  a component of $G_2 \setminus S$  to two    distinct  vertices, a contradiction. Thus    $5\le |S|\le 6$  and $S$ is not a clique.
Then  $\delta(G[S]) =d_{G[S]}(x)\leq |S| - 2$.
Since  $S\less x$ is a clique, we see that \[ d_1 = d_2 = |S| - 1 - d_{G[S]}(x)=|S| - 1 -\delta(G[S]). \]
It follows that  \[ e(G[S]) ={{|S|-1}\choose 2}+d_{G[S]}(x)= {{|S|-1}\choose 2}+\delta(G[S]). \] This, together with $(**)$, implies that 
	\begin{align*}
	9|S| &> 24 + 2(d_1 + d_2) + 2e(G[S]) \\
	&= 24 + 4|S| - 4 - 4\delta(G[S]) + (|S|-1)(|S|-2) + 2\delta(G[S]) \\
	&= 20+4|S|+|S|^2 -3|S| + 2 - 2\delta(G[S]) \\
	&\ge 20+ |S|^2 +|S| + 2 -2(|S|-2)\\
	&=|S|^2 -|S|+26, 
	\end{align*}
which is impossible because $5\le |S|\le 6$.
\end{proof}

\noindent {\bf Claim\refstepcounter{counter}\label{c:dnot6}   \arabic{counter}.}  
No vertex in $G$ is a $6$-vertex.
 
\begin{proof}
Suppose to the contrary that  $G$ has a $6$-vertex, say $x$.
By Claim~\ref{c:triangles}, $G[N(x)]$  contains  $K_{2,2,2}$ as a spanning subgraph.  Let $C$ be a component of $G \setminus N[x]$. By Claim~\ref{c:5conn}, $|N(C)|\ge5$. 
If  $N(C)=N(x)$ or $G[N(x)]$ contains  $K_6^=$ as a spanning subgraph, then $G\se\kef$ by contracting $C$  to a single vertex.   Thus $|N(C)\cap  N(x)|=5$  for every component $C$ of $G \setminus N[x]$ and $G[N(x)]=K_{2,2,2}$.  By Claim~\ref{c:triangles}, no vertex in $N(x)$ is a $5$-vertex in $G$. Thus $G \setminus N[x]$ is disconnected. Let $C'\ne C$ be another component of $G \setminus N[x]$. 
By Claim~\ref{c:almostclique}, $G[N(C)]=K_5^=$ and $G[N(C')]=K_5^=$.
Let $y\in N(x) \cap  N(C)  $ such that $y$ is incident with a missing edge of $G[N(C)]$. Then $G\se\kef$ by contracting $C$ onto $y$ and $C'$  to a single vertex, a  contradiction. 
\end{proof}

\noindent {\bf Claim\refstepcounter{counter}\label{c:d8}   \arabic{counter}.}  
No vertex in $G$ is a $7$-vertex.
 
\begin{proof}
Suppose $G$ has a $7$-vertex, say $x$.
  By Claim~\ref{c:triangles}, $\delta(G[N(x)])\ge 4$,   and so  $\overline{G}[N(x)]$ is the disjoint union of paths and cycles.   By Claim~\ref{c:d8}, no vertex in $N(x)$ is a $6$-vertex in $G$. 
 Suppose  there exists  $y\in N(x)$ such that $y$ is a $5$-vertex in $G$. Then $N[y]\subseteq N[x]$ and  $G[N[y]]=K_6$. Thus $G[N[x]]\se\kef$ because $\delta(G[N(x)])\ge 4$, a contradiction.  This proves that no  vertex in $N(x)$   is a $5$-vertex or $6$-vertex  in $G$.  Let $\mathcal{F}:=\{C_7, C_6  \cup K_1, C_5  \cup K_2,  C_4  \cup K_3\}$. Then $\overline{G}[N(x)]$ is a spanning subgraph of  some member in  $\mathcal{F}$ because   $\overline{G}[N(x)]$ is the disjoint union of paths and cycles. 
 We only consider the worst scenario that $\overline{G}[N(x)]$ is isomorphic to some graph in $\mathcal{F}$. It is straightforward to check that every graph in $\mathcal{F}$ has a $\mathcal{K}_6^{-3}$ minor.  Let $C$ be a component of $G \setminus N[x]$. By Claim~\ref{c:5conn}, $|N(C)|\ge5$. If $N(C)=N(x)$, then $G\se \kef$ by contracting $C$  to a single vertex. Thus $N(C)\ne N(x)$ for every  component $C$ of $G \setminus N[x]$.   We claim that $G\less N[x]$ is disconnected. Suppose not. Then $C=G\less N[x]$ and $\overline{G}[N(x)]=K_1+C_6$. Note that every vertex on $C_6$ belongs to $N(C)$. Thus   $G\se\kef$ by contracting   $G\less N[x]$ onto a vertex of $C_6$, a contradiction.  This proves that   $G\less N[x]$ is disconnected, as claimed.  Let $C$ and $C'$ be two distinct components of  $G\less N[x]$ such that $N(C)$ contains the most number of missing edges of $G[N(x)]$.  It is straightforward to check that  either $G[N(C)]$ must contain two adjacent missing edges, say $x_1x_2, x_1x_3$, where   $x_1, x_2, x_3\in N(x)$ are distinct, or    $\overline{G}[N(x)]\in\{K_1\cup C_6, K_2\cup C_5\}$ and $G[N(C)]=G[N(C')]=K_5^=$. In the former case, since $|N(C')|\ge5$, by Claim~\ref{c:almostclique}, we see that $G[N(C')\less x_1]$ must contain a missing edge, say $e$.   By contracting $C$ onto $x_1$ and $C'$ onto an end of the edge $e$, it is straightforward to check that $G[N[x]]+\{x_1x_2, x_1x_3,  e\}\se\kef$, a contradiction. In the latter case, let $y_1y_2$ be a missing edge in $G[N(C)]$ and $y_3y_4$ be a missing edge in $G[N(C')]$ such that  $y_1y_2, y_3y_4$ are two different missing edges in $G[N[x]]$. We may assume that $y_1\ne y_3$. Then $G[N[x]]+\{y_1y_2, y_3y_4\}\se\kef$ by contracting $C$ onto $y_1$  and  $C'$ onto $y_3$, a contradiction.  \end{proof}

\noindent {\bf Claim\refstepcounter{counter}\label{c:k72}   \arabic{counter}.}  
  $G$ is  $\mathcal{K}_7^{-2}$-free.
 
\begin{proof}
Suppose $G$ has a subgraph $H$ such that $H\in \mathcal{K}_7^{-2}$.  Since $G$ is $5$-connected, 
 we obtain a $\kef$ minor in $G$ by contracting a component of $G \setminus V(H)$  to a single vertex, a contradiction. \end{proof}

\noindent {\bf Claim\refstepcounter{counter}\label{c:onlyone5}   \arabic{counter}.}  
$G$ has at most  one $5$-vertex.
 
\begin{proof}
Suppose to the contrary that $G$ has two distinct $5$-vertices, say $x$ and $y$.
By Claims~\ref{c:triangles} and \ref{c:55}, $xy \notin E(G)$ and $N[x]$ and $N[y]$ are  $6$-cliques.
Suppose $N(x) \neq N(y)$.
By Claim~\ref{c:5conn} and Menger's Theorem, there exist five pairwise internally vertex-disjoint $(x, y)$-paths $Q_1, \dotsc, Q_5$. We may assume that $Q_1$ has  at least four vertices  with  $ V(Q_1)\cap N(x)=\{x_1\}$ and $ V(Q_1)\cap N(y)=\{y_1\}$.  Let $Q_1^*$ be the $(x_1, y_1)$-subpath of $Q_1$. Then $G\se\kef$ by  contracting all the edges of $Q_1^*\less y_1$, $Q_2\less\{x, y\}, \ldots, Q_5\less\{x, y\}$,   a contradiction.   Thus $N(x) = N(y)$. Then $G[N[x] \cup N[y]]=K_7^-$, contrary to Claim~\ref{c:k72}.  
\end{proof}

\noindent {\bf Claim\refstepcounter{counter}\label{c:58}   \arabic{counter}.}  
No $8$-vertex   is adjacent to a $5$-vertex  in $G$.
 
\begin{proof}
Suppose to the contrary that there exists $xy \in E(G)$ such that $d(x) = 8$ and $d(y) = 5$.
By Claim~\ref{c:triangles},  $G[N[y]]=K_6$,    $N[y]\subseteq N[x]$ and 
$\delta(G[(N(x)])\ge4$. Let $A:=N(x)\less N[y]$. Then $|A|=3$ because $xy\in E(G)$.  Let  $A:=\{a_1, a_2, a_3\}$.  Suppose $G[A]\ne K_3$, say $a_1a_2 \notin E(G)$. Then either $e(\{a_1, a_2\}, N(x)\less A)\ge 7$, or $e(\{a_1, a_2\}, N(x)\less A)=6$ and $a_3$ is complete to $\{a_1, a_2\}$, because $\delta(G[(N(x)])\ge4$. Then $G[N[x]\less a_3]\se\kef$ in the former case, and $G[N[x] ]/a_1a_3\se\kef$  in the latter case, a contradiction.  This proves that $G[A]=K_3$. Then $e(a_i, N(x)\less A)\ge 2$ for each $i\in[3]$. If $e(a_j, N(x)\less A)\ge 3$ for some   $j\in[3]$, then $G[N[y]\cup\{a_j\}]$ is not $\mathcal{K}_7^{-2}$-free, contrary to Claim~\ref{c:k72}. Thus $e(a_i, N(x)\less A)= 2$ for each $i\in[3]$.  Next if  there exists $z\in N(x)\cap N(y)$ such that $z$ has exactly one neighbor, say $a_1$, in $A$, then $G[N[x]]/za_1\se\kef$, a contradiction. Since $e(A, N(x)\less A)= 6$ and $|N(x)\cap N(y)|=4$, it follows that there exists $z^*\in N(x)\cap N(y)$ such that $z^*$  is anticomplete to $A$ in $G$. By Claim~\ref{c:onlyone5}, $z^*$ is not a $5$-vertex in $G$. Let $C$ be a component of  $G\less N[x]$ such that $z^*\in N(C)$.  Note that $y\notin N(C)$ and $N(x)\cap N(y)$ is a $4$-clique. By Claim~\ref{c:almostclique}, $N(C)$ must contain at least two vertices, say $a_1, a_2$, in $A$.     By contracting $C$ onto $z^*$, we see that  $G[N[x]\less a_3]+\{z^*a_1, z^*a_2\}\se\kef$, a contradiction. 
\end{proof}

\noindent {\bf Claim\refstepcounter{counter}\label{c:disconnected}   \arabic{counter}.}   Let $x\in V(G)$ be an $8$-vertex in $G$, and let $M$ be  the set of vertices of $N(x)$ not adjacent to all other vertices of $N(x)$.  
Then there is no component $K$ of $G \setminus N[x]$    such that $M\subseteq N(K)$.
In particular, $G \setminus N[x]$ is  disconnected.

\begin{proof}
Suppose such a component $K$ exists. Then every vertex in $M$ has a neighbor in $K$ because $M\subseteq N(K)$. 
 By Lemma~\ref{l:computer}, there exists  $y \in N(x)$ such that $G[N(x)] \setminus y$ has a $\ksix$ minor.
If $N[y] = N[x]$, then $G[N[x]]\se\kef$, a contradiction.
Thus $y\in M\subseteq N(K)$. 
By contracting $K$ onto $y$, we  obtain  a $\kef$ minor in $G$, a contradiction. This proves that no  such  component $K$ exists.  By Claims~\ref{c:dnot6}, \ref{c:d8} and \ref{c:58}, every vertex in $M$ has degree at least eight, and thus  every vertex in $M$  has a neighbor in  $G \setminus N[x]$. It follows that $G \setminus N[x]$ is disconnected. 
\end{proof}

\noindent {\bf Claim\refstepcounter{counter}\label{c:comporder1}   \arabic{counter}.}  
Let $x\in V(G)$ be an $8$-vertex in $G$. 
Then $G \setminus N[x]$ has at most one component $C$ with $|C|=1$.
Moreover, if $C$ is a component of $G \setminus N[x]$ such that $|C| = 1$, then the sole vertex in $C$  is a $5$-vertex in $G$.

\begin{proof} Let $C$ be a component of  $G \setminus N[x]$   with $|C|=1$.   Let $y$ be the only vertex in $C$.  Suppose $y$ is not a $5$-vertex in $G$. Then $d(y)\ge 8$  by Claims~\ref{c:dnot6} and  \ref{c:d8}, and so   $N(C)=N(x)$, contrary to Claim~\ref{c:disconnected}.  Thus   $y$ is  a $5$-vertex in $G$. By Claim~\ref{c:onlyone5}, $G \setminus N[x]$ has at most one such component $C$.
\end{proof}

\noindent {\bf Claim\refstepcounter{counter}\label{c:comp9}    \arabic{counter}.}  
 Let $x\in V(G)$ be an $8$-vertex in $G$. 
Then there is no component $C$ of $G \setminus N[x]$ such that $|C| \geq 2$ and  for every vertex $y\in V(C)$,   either   $d_G(y) =5$ or    $d_G(y) \geq 9$.

\begin{proof}
Suppose such a component $C$ exists.
By Claim~\ref{c:onlyone5}, $C$ has at most one $5$-vertex in $G$. Let $y\in   V(C)$ such that $d_G(y)\ge 9$. Then $|V(C) \cup N(C)| \geq 10$. By Claim~\ref{c:disconnected}, $N(C)\ne N(x)$. Thus $|C|\ge |N[y]\less N(C)|\ge 10-7=3$. 
Let $G_1 := G \setminus C$ and $G_2 := G[C \cup N(C)]$. Note that  $|G_2|\ge10$ and $N(C)$ is a minimal separating set of $G$.  
Let  $d_1$   be defined as in the paragraph prior to Claim~\ref{c:nodis7}. 
Let $z \in N(C)$  such that  $d_{G[N(C)]}(z) = \delta(G[N(C)])$. Let $d:=d_{G[N(C)]}(z)$. 
By contracting   $G_1 \setminus N(C)$ onto $z$, we see that   $d_1 \geq |N(C)| - d - 1$. By ($*$), 
  \[ e(G_2) <4.5(|C| + |N(C)|) - 12 - (|N(C)| - d - 1)=4.5|C| + 3.5|N(C)| + d - 11. \tag{a} \]
Now let $t = e_G ( C, N(C) )$.
Then $e(G_2) = e(C) + t + e(G[N(C)])$.
Note that   $2e(C) \geq 9(|C| - 1) + 5 - t$ and $2e(G[N(C)]) \geq d|N(C)|$, so we have \[ 2e(G_2) =2e(C) +2 t + 2e(G[N(C)])\geq 9|C| - 4 + t + d|N(C)|. \tag{b} \]
Combining  (a) and (b) yields  \[ 9|C| + 7|N(C)| + 2d - 22 > 2e(G_2) \geq 9|C| - 4 + t + d|N(C)| \] and so \[ -t > d\big(|N(C)| - 2\big) - 7|N(C)| + 18.\tag{c} \]
Note  that $\delta(G[N(x)]) \geq 4$ by Claim~\ref{c:triangles}, and $N(C)$ is a subset of $N(x)$, so \[ d = \delta(G[N(C)]) \geq 4 - (8 - |N(C)|) = |N(C)| - 4. \]
This, together with (c),  implies that 
	\begin{align*}
	-t &> \big(|N(C)| - 4\big)\big(|N(C)| - 2\big) - 7|N(C)| + 18 \\
	&= |N(C)|^2 - 13|N(C)| + 26 \\
	&= \left( |N(C)|- \frac{13}{2} \right)^2 - \frac{65}{4},
	\end{align*}
so $-t \geq -16$.
But then  \[ |C|(|C|-1)\ge 2e(C) \geq 9(|C|-1) +5 - t \geq 9|C| - 20,\]  
 which implies that $|C| \geq 8$ because $|C|\ge3$.  Thus  $G\se C\se \kef$, contrary to the choice of $G$. 
\end{proof}

To complete the proof, since $e(G)=\lceil 4.5n-12\rceil$, we have $\delta(G)\le8$. By Claims~\ref{c:55}, \ref{c:dnot6} and \ref{c:d8}, let $x$ be an $8$-vertex  in $G$. By Claim~\ref{c:comporder1},  let $C$ be a component of $G \setminus N_G[x]$ with $|C| \geq 2$ and, subject to that, $x$ and $C$ are chosen so that $|C|$ is minimized.
By Claims~\ref{c:onlyone5} and~\ref{c:comp9}, $C$  contains an $8$-vertex, say  $y$, in $G$.  Note that $N_G(x)\ne N_G(y)$ because $|C|\ge2$.  Let $K$ be the component of $G \setminus N_G[y]$ containing $x$.  Then $|K|\ge 2$ because $N_G(x)\ne N_G(y)$.  Note that $N_G(x)\cap N_G(y)\subseteq N(K)$, and every vertex in $N_G(x)\less  N_G(y)$ belongs to $K$. Let  $M$ be  the set of vertices of $N_G(y)$ not adjacent to all other vertices of $N_G(y)$.  
We claim that $M\subseteq N(K)$.
Suppose not. Let $z \in M \setminus N(K)$.
Then  $z \notin N_G(x)$, else $z\in N(K)$  because $x\in V(K)$. It follows that  $z \in V(C)$. 
 By Claim~\ref{c:58}, we   have $d_G(z) \geq 8$; in addition, $N_G(z)$ has no $5$-vertex in $G$ if $d_G(z) = 8$, and at least one vertex in $N_G(z)\less N_G(y)$ is not a $5$-vertex in $G$ if $d_G(z) \geq 9$, due to Claim~\ref{c:onlyone5}. In either case,    let $z'$ be a neighbor of $z$ in $G \setminus N_G[y]$ such that $d_G(z')\ge 8$.
Note that $z' \in \big(N_G(x)\less N_G(y)\big) \cup V(C)$.   
Suppose $z' \notin V(K)$. Then $z'\in V(C)$ because every vertex in $N_G(x)\less  N_G(y)$ belongs to $K$. Let  $C'$ be the component of $G \setminus N_G[y]$ that contains $z'$. Then  $|C'|\ge2$ by Claim~\ref{c:comporder1}, and $C'$ is a proper subset of $C$, contrary to our choice of $x$ and $C$. This proves that  $z' \in V(K)$, and so  $z \in N(K)$, contrary to the choice of $z$.
Thus $K$ is a component of $G \setminus N_G[y]$ such that $M \subseteq N(K)$, contrary to Claim~\ref{c:disconnected}.   \medskip

This completes the proof of Theorem~\ref{t:exfun}.
\end{proof}


\begin{thebibliography}{99}

\bibitem{AG18}
Boris Albar and Daniel Gon\c{c}alves.
\newblock On triangles in {$K_r$}-minor free graphs.
\newblock {\em J. Graph Theory}, 88(1):154--173, 2018.

\bibitem{AH77}
K.~Appel and W.~Haken.
\newblock Every planar map is four colorable. {I}. {D}ischarging.
\newblock {\em Illinois J. Math.}, 21(3):429--490, 1977.

\bibitem{AHK77}
K.~Appel, W.~Haken, and J.~Koch.
\newblock Every planar map is four colorable. {II}. {R}educibility.
\newblock {\em Illinois J. Math.}, 21(3):491--567, 1977.

\bibitem{CRS11}
Maria Chudnovsky, Bruce Reed, and Paul Seymour.
\newblock The edge-density for {$K_{2,t}$} minors.
\newblock {\em J. Combin. Theory Ser. B}, 101(1):18--46, 2011.

\bibitem{CV2020}
Kathie Cameron and Kristina Vu\v{s}kovi\'{c}.
\newblock Hadwiger's conjecture for some hereditary classes of graphs: a
  survey.
  \newblock Bulletin of the European Association for Theoretical Computer Science,  131, 2020.

\bibitem{Dirac52}
G.~A. Dirac.
\newblock A property of {$4$}-chromatic graphs and some remarks on critical
  graphs.
\newblock {\em J. London Math. Soc.}, 27:85--92, 1952.

\bibitem{Dirac60}
G.~A. Dirac.
\newblock Trennende {K}notenpunktmengen und {R}eduzibilit{\" a}t abstrakter
  {G}raphen mit {A}nwendung auf das {V}ierfarbenproblem.
\newblock {\em J. Reine Agew. Math.}, 204:116--131, 1960.

\bibitem{Dirac64b}
G.~A. Dirac.
\newblock Homomorphism theorems for graphs.
\newblock {\em Math. Ann.}, 153:69--80, 1964.

\bibitem{Dirac64a}
G.~A. Dirac.
\newblock On the structure of {$5$}- and {$6$}-chromatic abstract graphs.
\newblock {\em J. Reine Angew. Math.}, 214(215):43--52, 1964.

\bibitem{DelcourtPostle}
Michelle Delcourt and Luke Postle.
\newblock Reducing {L}inear {H}adwiger's {C}onjecture to coloring small graphs.
\newblock arXiv:2108.01633.

\bibitem{Had43}
H.~Hadwiger.
\newblock \"{U}ber eine {K}lassifikation der {S}treckenkomplexe.
\newblock {\em Vierteljschr. Naturforsch. Ges. Z\"{u}rich}, 88:133--142, 1943.

\bibitem{Jakobsen71a}
I.~T. Jakobsen.
\newblock A homomorphism theorem with an application to the conjecture of
  {H}adwiger.
\newblock {\em Studia Sci. Math. Hungar.}, 6:151--160, 1971.

\bibitem{Jakobsen71b}
I.~T. Jakobsen.
\newblock Weakening of the conjecture of {H}adwiger for $8$- and $9$-chromatic
  graphs.
\newblock {\em Aarhus Universitet, Matematisk Institut Preprint Series},
  22:1--17, 1971.

\bibitem{Jor01}
Leif~K. J\o rgensen.
\newblock Vertex partitions of {$K_{4,4}$}-minor free graphs.
\newblock {\em Graphs Combin.}, 17(2):265--274, 2001.


\bibitem{KT05}
Ken-ichii Kawarabayashi and Bjarne Toft.
\newblock Any 7-chromatic graph has ${K}_7$ or ${K}_{4,4}$ as a minor.
\newblock {\em Combinatorica}, 25:327--353, 2005.

\bibitem{K2015}
Ken-ichi Kawarabayashi.
\newblock Hadwiger's conjecture.
\newblock In {\em Topics in chromatic graph theory}, volume 156 of {\em
  Encyclopedia Math. Appl.}, pages 73--93. Cambridge Univ. Press, Cambridge,
  2015.

\bibitem{Kostochka82}
Alexandr~V. Kostochka.
\newblock The minimum {H}adwiger number for graphs with a given mean degree of
  vertices.
\newblock {\em Metody Diskret. Analiz.}, (38):37--58, 1982.

\bibitem{Kostochka84}
Alexandr~V. Kostochka.
\newblock Lower bound of the {H}adwiger number of graphs by their average
  degree.
\newblock {\em Combinatorica}, 4(4):307--316, 1984.

\bibitem{KosPri08}
Alexandr~V. Kostochka and Noah Prince.
\newblock On {$K_{s,t}$}-minors in graphs with given average degree.
\newblock {\em Discrete Math.}, 308(19):4435--4445, 2008.

\bibitem{KosPri10}
Alexandr~V. Kostochka and Noah Prince.
\newblock Dense graphs have {$K_{3,t}$} minors.
\newblock {\em Discrete Mathematics}, 310(20):2637--2654, 2010.


\bibitem{Kos14}
A.~V. Kostochka.
\newblock {$K_{s,t}$} minors in {$(s+t)$}-chromatic graphs, {II}.
\newblock {\em J. Graph Theory}, 75(4):377--386, 2014.

\bibitem{KuhOst03c}
Daniela K\"{u}hn and Deryk Osthus.
\newblock Minors in graphs of large girth.
\newblock {\em Random Structures Algorithms}, 22(2):213--225, 2003.

\bibitem{KuhOst05a}
Daniela K{\"u}hn and Deryk Osthus.
\newblock Forcing unbalanced complete bipartite minors.
\newblock {\em European Journal of Combinatorics}, 26(1):75--81, 2005.


\bibitem{7con}
W.~Mader.
\newblock \"{U}ber trennende {E}ckenmengen in homomorphiekritischen {G}raphen.
\newblock {\em Math. Ann.}, 175:243--252, 1968.

\bibitem{Myers03}
Joseph~Samuel Myers.
\newblock The extremal function for unbalanced bipartite minors.
\newblock {\em Discrete mathematics}, 271(1-3):209--222, 2003.

\bibitem{NPS20}
Sergey Norin, Luke Postle, and Zi-Xia Song.
\newblock Breaking the degeneracy barrier for coloring graphs with no {$K_t$}
  minor.
\newblock arXiv:12206.001862.

\bibitem{NorSey22}
Sergey Norin and Paul Seymour.
\newblock Dense minors of graphs with independence number two.
\newblock arXiv:1910.09378v2.

\bibitem{Rolek20}
Martin Rolek.
\newblock Graphs with no {$K^=_9$} minor are 10-colorable.
\newblock {\em J. Graph Theory}, 93(4):560--565, 2020.

\bibitem{RolekSong17a}
Martin Rolek and Zi-Xia Song.
\newblock Coloring graphs with forbidden minors.
\newblock {\em J. Combin. Theory Ser. B}, 127:14--31, 2017.

\bibitem{RST22}
Martin Rolek, Zi-Xia Song, and Robin Thomas.
\newblock Properties of {$8$}-contraction-critical graphs with no {$K_7$}
  minor.
\newblock arXiv:2208.07335.

\bibitem{RST}
Neil Robertson, Paul Seymour, and Robin Thomas.
\newblock Hadwiger's conjecture for {$K_6$}-free graphs.
\newblock {\em Combinatorica}, 13(3):279--361, 1993.

\bibitem{Seymoursurvey}
P.~Seymour.
\newblock {\em Hadwiger's {C}onjecture}, chapter~13, pages 417--437.
\newblock Springer, Cham, 2016.
\newblock In: Open Problems in Mathematics (edited by J. Nash Jr. and M.
  Rassias).

\bibitem{SongThomas2006}
Zi-Xia Song and Robin Thomas.
\newblock The extremal function for {$K_9$} minors.
\newblock {\em J. Combin. Theory Ser. B}, 96(2):240--252, 2006.

\bibitem{Thomason84}
Andrew Thomason.
\newblock An extremal function for contractions of graphs.
\newblock {\em Math. Proc. Cambridge Philos. Soc.}, 95(2):261--265, 1984.

\bibitem{Wagner37}
K.~Wagner.
\newblock \"{U}ber eine {E}igenschaft der ebenen {K}omplexe.
\newblock {\em Mathematische Annalen}, 114:570--590, 1937.

\bibitem{Wagner60}
K.~Wagner.
\newblock Bemerkungen zu {H}adwiger's {V}ermutung.
\newblock {\em Mathematische Annalen}, 141:433--451, 1960.

\bibitem{Woo01}
Douglas~R. Woodall.
\newblock List colourings of graphs.
\newblock In {\em Surveys in combinatorics, 2001 ({S}ussex)}, volume 288 of
  {\em London Math. Soc. Lecture Note Ser.}, pages 269--301. Cambridge Univ.
  Press, Cambridge, 2001.

\end{thebibliography}
\end{document}